\documentclass[11pt,reqno]{amsart}
\usepackage[table]{xcolor}
\usepackage{amssymb, amsmath, amsthm, mathrsfs}
\usepackage{enumitem}
\usepackage{float}
\usepackage{array}
\usepackage{booktabs}
\usepackage{tabularx}
\usepackage{longtable}
\usepackage{multirow}
\usepackage{colortbl}
\usepackage{makecell}
\usepackage{pgfplots} 
\pgfplotsset{compat=1.18}
\usepackage{graphicx}
\usepackage{adjustbox}
\usepackage{placeins}
\usepackage{tikz}
\usepackage{float}

\usetikzlibrary{
	arrows.meta,
	positioning,
	calc,
	shapes.geometric,
	shapes.multipart,
	fit,
	matrix,
	backgrounds,
	decorations.pathreplacing,
	decorations.markings
}

\usepackage{hyperref}
\hypersetup{
	colorlinks=true,
	linkcolor=black,
	citecolor=black,
	urlcolor=blue,
	pdftitle={Finite-Order Entire Solutions to Fermat-Type Partial Differential-Difference Systems in Cn},
	pdfauthor={Sujoy Majumder, Jhilik Banerjee, Abhijit Banerjee, and Elias G. Saleeby}
}

\newtheorem*{theoA}{Theorem A}
\newtheorem*{theoB}{Theorem B}
\newtheorem*{theoC}{Theorem C}
\newtheorem*{theoD}{Theorem D}
\newtheorem*{theoE}{Theorem E}

\newtheorem*{corE}{Corollary E}

\newtheorem{theo}{Theorem}[section]
\newtheorem{lem}{Lemma}[section]

\newtheorem{exm}{Example}[section]

\newtheorem{rem}{Remark}[section]

\newcommand{\ol}{\overline}

\newcommand{\be}{\begin{equation}}
	\newcommand{\ee}{\end{equation}}

\newcommand{\beas}{\begin{eqnarray*}}
	\newcommand{\eeas}{\end{eqnarray*}}

\newcommand{\bea}{\begin{eqnarray}}
	\newcommand{\eea}{\end{eqnarray}}

\numberwithin{equation}{section}

\begin{document}

\title[Finite-Order Entire Solutions to Fermat-Type Systems]{Finite-Order Entire Solutions to Fermat-Type Partial Differential-Difference Systems in $\mathbb{C}^n$}

\date{}
\author[ S. M\MakeLowercase {ajumder}, J. B\MakeLowercase{anerjee},  A. B\MakeLowercase{anerjee} \MakeLowercase{ and } E. G. S\MakeLowercase{aleeby}]{ S\MakeLowercase {ujoy} M\MakeLowercase {ajumder}$^1$, J\MakeLowercase{hilik} B\MakeLowercase{anerjee}$^2$,  A\MakeLowercase {bhijit} B\MakeLowercase {anerjee}$^3$\MakeLowercase{ and } E\MakeLowercase{lias} G. S\MakeLowercase{aleeby}$^4$$^*$}
\address{$^{1}$Department of Mathematics, Raiganj University, Raiganj, West Bengal-733134, India.}
\email{sm05math@gmail.com; sjm@raiganjuniversity.ac.in}
\address{$^2$Department of Mathematics, University of Kalyani, West Bengal 741235, India.}
\email{jhilikbanerjee38@gmail.com; jhilikmath24@klyuniv.ac.in}
\address{$^3$Department of Mathematics, University of Kalyani, West Bengal 741235, India.}
\email{abanerjee\_kal@yahoo.co.in}
\address{$^4$Department of Mathematics, Al Akhawayn University, Ifrane, 53000, Morocco.}
\email{esaleeby@yahoo.com; e.saleeby@aui.ma}

\renewcommand{\thefootnote}{}
\footnote{2020 Mathematics Subject Classification: 39A45, 32H30, 39A14 and 35A20.}
\footnote{Key words and phrases: Several complex variables, meromorphic functions, Fermat-type
equations, Nevanlinna theory, partial differential-difference equations.}
\footnote{*Corresponding Author: Elias G. Saleeby.}

\renewcommand{\thefootnote}{\arabic{footnote}}
\setcounter{footnote}{0}

\begin{abstract} The primary objective of this paper is to determine the existence and explicit form of finite-order entire solutions in $\mathbb{C}^n$ of the following system of Fermat-type partial differential-difference equations:
\[\begin{cases}
\left(\frac{\partial f_1\left(z\right)}{\partial z_1}\right)^{n_1} + (f_2 \left(z+c\right)-f_1(z) )^{m_1}= 1,\\
		
\medskip
\left(\frac{\partial f_2\left(z\right)}{\partial z_1}\right)^{n_2} + (f_1 \left(z+c \right)-f_2(z) )^{m_2}= 1,
\end{cases}\]	
for several choices of the positive integers $n_1$, $n_2$, $m_1$, and $m_2$, where $c=(c_1,c_2,\ldots,c_n)$. We obtain structural classifications and nonexistence results in the exponent regimes specified in the main theorems, extending results of Xu et al. \cite{XLL1} from $\mathbb{C}^2$ to $\mathbb{C}^n$. Several examples illustrate the resulting solution families.
	
\end{abstract}

\thanks{Typeset by \AmS -\LaTeX}
\maketitle

\section{{\bf Introduction}}
The equation $x^n+y^m=1,$ $m,n \in \mathbb{N}$ commonly known as the \emph{Fermat equation}, occupies a central position in the development of number theory and modern mathematics. The celebrated proof of Fermat's Last Theorem by A. Wiles and R. Taylor in 1995 \cite{Taylor+Wiles_AnnMath_1995,Wiles_AnnMath_1995} established that the equation admits no nontrivial rational solutions for $m=n\geq3$, whereas infinitely many nontrivial rational solutions exist in the quadratic case $m=n=2$. This landmark result motivated extensive study of analytic analogues of Fermat-type equations in complex analysis, difference equations and differential equations.

 The functional equation $f ^m(z)+ g^n(z) = 1$ can be regarded as the Fermat-type functional equation. It has drawn considerable attention from many mathematicians in the study of Fermat-type equations. It is known that the Fermat-type equation admits no transcendental meromorphic solutions for $n=m\geq 4$ (\cite{FG2}), and no transcendental entire solutions for $n=m\geq 3$ (\cite{M1}).

First, we recall a result established by Liu et al. \cite{LCC1}:
\begin{theoA}\cite{LCC1}
The transcendental entire solutions with finite-order of
\[f'(z)^2+[f(z+c)-f(z)]^2=1,\]
must satisfy
$f(z)=\frac{1}{2}\sin(2z+Bi),$
where $c=(2k+1)\pi,$ $k$ is an integer and $B$ is a constant.
\end{theoA}
Motivated by these developments, researchers began investigating coupled systems of Fermat-type equations. Recall that the pair $(f(z),g(z))$ is called a set of finite-order transcendental entire solutions for the system

\[\left\{\begin{array}{l}
f^{n_1}(z)+ g^{m_1}(z)=1 \\
f^{n_2}(z)+ g^{m_2}(z)=1
\end{array}\right.\]
if $f(z)$ and $g(z)$ are transcendental entire functions satisfying $\rho=\max \{\rho(f), \rho(g)\}<\infty$.

In 2016, Gao \cite{L.G} addressed the problem of determining the form of solutions for the following system of Fermat-type differential-difference equations
\bea\label{e1.1}
\left\{\begin{array}{l}
{\left(f_1^{(1)}(z)\right)^2+{f_2}^2(z+c)=1} \\
{\left(f_2^{(1)}(z)\right)^2+{f_1}^2(z+c)=1}
\end{array}\right.\eea
and derived the following result:

\begin{theoB}\cite[Theorem 1.1]{L.G}
Suppose that $(f_1(z), f_2(z))$ is a pair of finite order transcendental entire solutions for the system of differential-difference equations {\em(\ref{e1.1})}. Then $(f_1(z), f_2(z))$ satisfies one of the following
\begin{enumerate}
\item[\em{(i)}] $(f_1(z), f_2(z))=\left(\sin (z-b\iota), \sin (z-b_1 \iota)\right)$,
\item[\em{(ii)}] $(f_1(z), f_2(z))=\left(\sin (z+b\iota), \sin (z+b_1 \iota)\right)$,
\end{enumerate}
where $b, b_1$ are constants, and $c=k \pi$, $k$ is an integer.
\end{theoB}

\smallskip
We define $\mathbb{Z}_+=\mathbb{Z}[0,+\infty)=\{n\in \mathbb{Z}: 0\leq n<+\infty\}$ and $\mathbb{Z}^+=\mathbb{Z}(0,+\infty)=\{n\in \mathbb{Z}: 0<n<+\infty\}$. On $\mathbb{C}^n$, we define
\[\displaystyle \partial_{z_i}(f(z))=\frac{\partial f(z)}{\partial z_i},\ldots,\partial^{l_i}_{z_i}(f(z))=\frac{\partial^{l_i} f(z)}{\partial z_i^{l_i}},\;\;\partial^{I}(f(z))=\frac{\partial^{|I|}f(z)}{\partial z_1^{i_1}\cdots \partial z_n^{i_n}}\]
where $l_i\in \mathbb{Z}_+$, $i=1,2,\ldots,n$ and $I=(i_1,i_2,\ldots,i_n)\in\mathbb{Z}^n_+$ such that $|I|=\sum_{j=1}^n i_j$.

\smallskip
In recent years, Fermat-type equations involving difference operators and partial differential operators have attracted increasing attention due to their intrinsic analytical complexity and close connection with multidimensional Nevanlinna theory. Numerous contributions in this direction can be found in \cite{Haldar_Mediterian_2023,L2,L1,MSP,GS1,XC1,XC2,XLL1,XW1,XH,XuHaldar2024,HaldarBanerjeParvin2026,XuMajumderPramanik2025}. These investigations rely heavily on higher-dimensional Nevanlinna theory and the value distribution of meromorphic functions in several complex variables; see \cite{TBC1,CK1,CX1,K1}. The correction \cite{XC2} should be read together with the original paper \cite{XC1}.

Particularly noteworthy is the work of Haldar \cite{Haldar_Mediterian_2023}, who considered the Fermat-type partial differential-difference equation \bea\label{e1.2} (\partial_{z_1}(f(z_1,z_2)))^2+(f(z_1+c_1,z_2+c_2)-f(z_1,z_2))^2=1,\eea and proved the following theorem:
\begin{theoC} \cite[Theorem 1.4]{Haldar_Mediterian_2023}
	Let \(c=(c_1,c_2)\) be a constant in \(\mathbb{C}^2\). Then any transcendental entire
	solution with finite order of the Fermat-type partial differential-difference equation
	(\ref{e1.2}) is of the form
	\[	f(z_1,z_2)	=-\frac{1}{4i}\left(\beta_1 e^{\alpha\left(-\frac{2i}{\alpha}z_1+Bz_2+h(z_2)\right)}-\beta_2e^{-\alpha\left(-\frac{2i}{\alpha}z_1+Bz_2+h(z_2)\right)}\right),\]
	where \(B\), \(\alpha\), \(\beta_1\), \(\beta_2\) are complex constants with
	\(\beta_1\beta_2=1\), and \(h\) is a polynomial in \(z_2\) only satisfying the relation
	\[h(z_2)-h(z_2+c_2)=-\frac{2ic_1}{\alpha}+	Bc_2+\frac{1}{\alpha}\log(-1).\]
	
	In the special case when \(c_2\neq 0\), \(h\) must be constant and the solution is of the form
	\[f(z_1,z_2)=-\frac{1}{4i}\left(\beta_1 e^{\alpha\left(-\frac{2i}{\alpha}z_1+Bz_2+C\right)}	-\beta_2 e^{-\alpha\left(-\frac{2i}{\alpha}z_1+Bz_2+C\right)}\right),	\]
	where \(B\) and \(C\) are complex constants such that
	\[	e^{2ic_1-\alpha Bc_2}=-1.	\]
\end{theoC}

The following theorem describes the entire and meromorphic solutions in $\mathbb{C}^n$ of the Fermat-type Eq.
\bea\label{FT1} f^m+g^m=1,\;m>1.\eea

\begin{theoD} \cite[Theorem 1.3]{GS2} For $h:\mathbb{C}^n \to \mathbb{C}$ entire, the solutions of the Eq. {\em(\ref{FT1})} are characterized as follows:
\begin{enumerate}
\item[\em{(a)}] for $m=2$, the entire solutions are $f=\cos(h)$ and $g=\sin(h)$;
\item[\em{(b)}] for $m>2$, there are no non-constant entire solutions;
\item[\em{(c)}] for $m=2$, the meromorphic solutions are of the form
$f=\frac{1 - \beta^2}{1 + \beta^2}$ and $g=\frac{2\beta}{1 + \beta^2}$,
with $\beta$ being meromorphic on $\mathbb{C}^n$;
\item[\em{(d)}] for $m=3$, the meromorphic solutions are of the form
\[f = \frac{1}{2\wp(h)}\left( 1 + \frac{\wp^{(1)}(h)}{\sqrt{3}}\right)\;\;\text{and}\;\;g = \frac{1}{2\wp(h)} \left(1-\frac{\wp^{(1)}(h)}{\sqrt{3}}\right);\] where $\wp(\cdot)=\wp(\cdot;g_2,g_3)$ denotes the Weierstrass elliptic function
satisfying $(\wp')^2=4\wp^3-g_2\wp-g_3$.
\item[\em{(e)}] for $m>3$, there are no non-constant meromorphic solutions.
\end{enumerate}
\end{theoD}

\smallskip
Recently Xu et al. \cite{XLL1} studied the existence and the forms of the finite order transcendental entire solutions to the following system of Fermat-type partial differential-difference equations:
\bea\label{pdds1.4}
\left\{
\begin{aligned}
\left(\partial_{z_1} f_1(z_1,z_2)\right)^2
+\left[f_2(z_1+c_1,z_2+c_2)-f_1(z_1,z_2)\right]^2 &=1,\\
\left(\partial_{z_1} f_2(z_1,z_2)\right)^2
+\left[f_1(z_1+c_1,z_2+c_2)-f_2(z_1,z_2)\right]^2 &=1.
\end{aligned}
\right.\eea
\smallskip
\begin{theoE} \cite{XLL1} Let $c =(c_1, c_2)$ be a constant in $\mathbb{C}^2$. Let $(f_1, f_2)$ be a pair of transcendental entire solutions with finite order for the system of Fermat-type partial differential-difference equations {\em(\ref{pdds1.4})}.
\begin{enumerate}
\item[\em{(i)}] If $\frac{\partial f_1}{\partial z_1}$ or 
$\frac{\partial f_2}{\partial z_1}$ is a constant, then 
$f_1,f_2$ have the following forms
\[(f_1(z),f_2(z))=\left(G(e^{\delta_1(z_2)},\ldots,e^{\delta_q(z_2)})+L_1(z)+D_1,\,
\eta G(e^{\delta_1(z_2)},\ldots,e^{\delta_q(z_2)})+L_1(z)+D_2\right),\]
where
$G(x_1,\ldots,x_q)=\mu_1x_1+\cdots+\mu_qx_q,$ $\delta_j(z_2)=\theta_j z_2,$
$\mu_j,\theta_j\;(j=1,\ldots,q)$ are constants in $\mathbb C$,
$L_1(z)=b_1z_1+b_2z_2,$
$D_1,D_2$ satisfy
$b_1^2+(b_1c_1+b_2c_2)^2=1,\; D_1=D_2,$
or
$b_1^2+(D_1-D_2)^2=1,\;\; L_1(c)=b_1c_1+b_2c_2=0,$
and $c_2,\theta_j\;(j=1,\ldots,q),\eta$ satisfy one of the following
cases:
\begin{enumerate}
\item[\em{(i$_1$)}] $\eta=1, \theta_j c_2=2k\pi i,\quad j=1,2,\ldots,q;$
\item[\em{(i$_2$)}] $\eta=-1, \theta_j c_2=(2k+1)\pi i,\quad j=1,2,\ldots,q.$
\end{enumerate}
		
\item[\em{(ii)}] If $\frac{\partial f_1}{\partial z_1}$ or 
$\frac{\partial f_2}{\partial z_1}$ is not a constant, then
$f_1,f_2$ have the following forms
\[(f_1(z),f_2(z))=(f(z)+D_1,\eta f(z)+D_2),\]
where
\[f(z)=A_{11}e^{L(z)}+A_{12}e^{-L(z)}=A_{11}e^{a_1z_1+a_2z_2}+A_{12}e^{-(a_1z_1+a_2z_2)},		\]
and $a_1,a_2,c,\eta,A_{11},A_{12},D_1,D_2$ are constants
satisfying one of the following cases:
\begin{enumerate}
\item[\em{(ii$_1$)}]
$a_1=-2\iota,\;\; 16A_{11}A_{12}=1,\;\; L(c)=2k\pi i,\;\;	\eta=-1,\;\; D_1=D_2;$
\item[\em{(ii$_2$)}]
$a_1=-2\iota,\;\; 16A_{11}A_{12}=1,\;\; L(c)=(2k+1)\pi i,\;\;	\eta=1,\;\; D_1=D_2.$
\end{enumerate}
\end{enumerate}
\end{theoE}
\begin{corE}
	
Let $(f_1,f_2)$ be any pair of transcendental entire solutions with finite order for the system of Fermat-type partial differential equations
\[	\left\{
\begin{aligned}
\left(\partial_{z_1} f_1(z_1,z_2)\right)^2
+\left[f_2(z_1,z_2)-f_1(z_1,z_2)\right]^2 &=1,\\
\left(\partial_{z_1} f_2(z_1,z_2)\right)^2
+\left[f_1(z_1,z_2)-f_2(z_1,z_2)\right]^2 &=1.
\end{aligned}
\right.\]
	
Then $f_1,f_2$ have the following forms
\begin{enumerate}
\item[\em{(i)}]
$(f_1(z),f_2(z))=(f(z)+D_1,\;f(z)+D_2),$
where
$f(z)=\mu_1e^{\theta_1 z_2}+\cdots+\mu_qe^{\theta_q z_2}+b_1z_1+b_2z_2,$
$\mu_j,\theta_j\;(j=1,\ldots,q),b_1,b_2,D_1,D_2$ are constants in $\mathbb{C}$ satisfying	$b_1^2+(D_1-D_2)^2=1;$
\item[\em{(ii)}]
$(f_1(z),f_2(z))=(f(z)+D_1,\,-f(z)+D_1),$
where
	$f(z)=	A_1e^{a_1z_1+a_2z_2}+A_2e^{-(a_1z_1+a_2z_2)},$
and $a_1,a_2,A_1,A_2,D_1$ are constants in $\mathbb{C}$
satisfying
\[
a_1=-2\iota,\qquad 16A_1A_2=1,\qquad D_1\in\mathbb{C}.
\]
\end{enumerate}
\end{corE}

Let us consider an example that will show that the conclusion of Theorem E and Corollary E does not cover all solutions.
\begin{exm}
Let \[(f_1(z),f_2(z))=\left(\frac{\cos (2z_1)}{2}+Q(z_2), -\frac{\cos (2z_1)}{2}+Q(z_2)\right),\] where $Q(z_2)$ is a polynomial in the variable $z_2$ and $c=(c_1,c_2)=(k\pi,0)$, with $k\in\mathbb{Z}$. Then $\cos(2(z_1+c_1))=\cos(2z_1)$, and direct substitution shows that $(f_1(z),f_2(z))$ is a solution of {\em(\ref{pdds1.4})}.
\end{exm}

\smallskip
Determining the precise form of finite-order transcendental entire solutions $(f_1(z),f_2(z))$ of the system (\ref{pdds1.4}) would therefore be intriguing. Motivated by this, we further aim to extend Theorem E by considering the following system of Fermat-type partial differential-difference equations in $\mathbb{C}^n$:
\bea\label{pds1}
\begin{cases}
\left(\partial_{z_1}(f_1(z))\right)^{n_1} + (f_2(z+c) -f_1(z))^{m_1}= 1, \\

\medskip
\left(\partial_{z_1}(f_2(z))\right)^{n_2} + (f_1(z+c) -f_2(z))^{m_2} = 1,
\end{cases}
\eea
where $n_1, n_2, m_1, m_2$ are positive integers and $c=(c_1, c_2, \ldots, c_n)\in\mathbb{C}^n$.

For $n_1=n_2=m_1=m_2=1$, it is easy to verify first considering  \begin{align*}
(f_1(z),f_2(z))=\left(e^{2(z_1+z_2+\ldots+z_n)}+\frac{z_1}{1+c_1},e^{2(z_1+z_2+\ldots+z_n)}+\frac{z_1}{1+c_1}\right),
\end{align*}
 where $c_1\ne -1$ and $e^{2(c_1+c_2+\ldots+c_n)}=-1$ and next considering \begin{align*}
(f_1(z),f_2(z))=\left(e^{2(z_1+z_2+\ldots+z_n)}+z_1^2,e^{2(z_1+z_2+\ldots+z_n)}+z_1^2\right),
\end{align*}
 where $c_1=-1$  and $e^{2(c_1+c_2+\ldots+c_n)}=-1$ are the solutions of Eq. (\ref{pds1}). Therefore, we consider the Eq. (\ref{pds1}) for the existence of entire solutions for the case when $n_i+m_i>2$, where $i=1,2$.

\subsection{\bf Basic Notations in several complex variables}

Nevanlinna theory, originally developed for meromorphic functions of a single complex variable, has been extensively extended to the setting of several complex variables to investigate the value distribution of meromorphic mappings from $\mathbb{C}^n$ into complex projective spaces and complex manifolds. This theory provides deep insights into: (a) the growth and value distribution of meromorphic mappings; (b) uniqueness and rigidity phenomena for solutions of linear and nonlinear partial differential equations; and (c) criteria for compactness and convergence of families of meromorphic mappings. Its applications encompass a wide range of areas, including complex analysis, partial differential equations, complex geometry, and mathematical physics, offering both powerful analytical techniques and elegant theoretical results. The references \cite{Banerjee-Majumder-2026}, \cite{LB}, \cite{FL}, \cite{Majumder-2026}, \cite{Majumder-Das}, \cite{Majumder-Das-Pramanik-2025}, \cite{Majumder-Sarkar}, \cite{Majumder-Sarkar-2026} and \cite{Majumder-Sarkar-2027} provide a comprehensive foundation for understanding the current state of research in Nevanlinna value distribution theory in several complex variables.
 
\smallskip
On $\mathbb{C}^n$, the exterior derivative $d$ splits $d= \partial+ \bar{\partial}$ and twists to $d^c= \frac{i}{4\pi}\left(\bar{\partial}- \partial\right)$ (see \cite{HLY1,WS1}). Clearly $dd^{c}= \frac{i}{2\pi}\partial\bar{\partial}$.
An exhaustion $\tau_n$ of $\mathbb{C}^n$ is defined by $\tau_n(z)=||z||^2$. The standard Kaehler metric on $\mathbb{C}^n$ is given by $\upsilon_n=dd^c\tau_n>0$. On $\mathbb{C}^n\backslash \{0\}$, we define $\omega_n=dd^c\log \tau_n\geq 0$ and $\sigma_n=d^c\log \tau_n \wedge \omega_n^{n-1}$. For any $S\subseteq \mathbb{C}^n$, let $S[r]$, $S(r)$ and $S\langle r\rangle$ be the intersection of $S$ with
respectively the closed ball, the open ball, the sphere of radius $r>0$ centered at $0\in \mathbb{C}^n$.

\smallskip
The zero multiplicity $\mu^0_f(a)$ of an entire function $f$ in $\mathbb{C}^n$ at a point $a\in \mathbb{C}^n$ is defined
to be the order of vanishing of $f$ at $a$. If $a=(a_1,\ldots,a_n)$, then $\partial^I(f(a))=0$, where $|I|\leq \mu^0_f(a)-1$.
In other words, we can write $f(z)=\sum_{i=0}^{\infty}P_i(z-a)$, where the term $P_i(z-a)$ is either identically zero or a homogeneous polynomial of degree $i$. Certainly $\mu^0_f(a)=\min\{i:P_i(z-a)\not\equiv 0\}$.

\smallskip
Let $f$ be a meromorphic function in $\mathbb{C}^n$. Then there exist holomorphic functions $g$ and $h$ such that $hf=g$ on $\mathbb{C}^n$ and $\dim_z h^{-1}(\{0\})\cap g^{-1}(\{0\})\leq n-2$. Therefore the $c$-multiplicity of $f$ is just $\mu^c_f=\mu^0_{g-ch}$ if $c\in\mathbb{C}$ and $\mu^c_f=\mu^0_h$ if $c=\infty$. The function $\mu^c_f: \mathbb{C}^n\to \mathbb{Z}$ is nonnegative and is called the $c$-divisor of $f$.
If $f\not\equiv 0$ on each component of $\mathbb{C}^n$, then $\nu=\mu_f=\mu^0_f-\mu^{\infty}_f$ is called the divisor of $f$.

\smallskip
Let $f$, $g$ and $a$ be meromorphic functions on $\mathbb{C}^n$. Then one can find three pairs
of entire functions $f_1$ and $f_2$, $g_1$ and $g_2$, and $a_1$ and $a_2$, in which each pair is coprime
at each point in $\mathbb{C}^n$ such that $f = f_2/f_1$, $g=g_2/g_1$ and $a = a_2/a_1$.
We say that $f$ and $g$ share $a$ CM if $\mu_{a_1f_2-a_2f_1}^0=\mu_{a_1g_2-a_2g_1}^0\;(a\not\equiv \infty)$ and $\mu_{f_1}^0=\mu_{g_1}^0\;\;(a=\infty)$. Again we say that $f$ and $g$ share $a$ IM if $\mu_{a_1f_2-a_2f_1,1}^0=\mu_{a_1g_2-a_2g_1,1}^0\;(a\not\equiv \infty)$ and $\mu_{f_1,1}^0=\mu_{g_1,1}^0\;\;(a=\infty)$.

\smallskip
For $t>0$, the counting function $n_{\nu}$ is defined by
\beas n_{\nu}(t)=t^{-2(n-1)}\int_{A[t]}\nu \upsilon_n^{n-1},\eeas
where $A=\text{supp}\;\nu=\ol{\{z\in G: \nu_f(z)\neq 0\}}$.
The valence function of $\nu$ is defined by
\[N_{\nu}(r)=N_{\nu}(r,r_0)=\int_{r_0}^r n_{\nu}(t)\frac{dt}{t}\;\;(r\geq r_0).\]

\smallskip
We write $N_{\mu_f^a}(r)=N(r,a;f)$ if $a\in\mathbb{C}$ and $N_{\mu_f^a}(r)=N(r,f)$ if $a=\infty$. For $k\in\mathbb{N}$, define the truncated multiplicity functions on $\mathbb{C}^n$ by $\mu_{f,k}^a(z)=\min\{\mu_f^a(z),k\}$, $\mu_{f(k}^a(z)=\mu_f^a(z)$ if $\mu_f^a(z)\geq  k$ and $\mu_{f(k}^a(z)=0$ if $\mu_f^a(z)<k$. Also we define the truncated valence functions $N_{\nu}(r)=\ol N(r,a;f)$ if $\nu=\mu_{f,1}^a$ and $N_{\nu}(r)=N_{(k}(r,a;f)$, if $\nu=\mu_{f(k}^a$.

\smallskip
An algebraic subset $X$ of $\mathbb{C}^n$ is defined as a subset
\[X=\left\lbrace z\in\mathbb{C}^n: P_j(z)=0,\;1\leq j\leq l\right\rbrace\]
with finitely many polynomials $P_1(z),\ldots, P_l(z)$.
A divisor $\nu$ on $\mathbb{C}^n$ is said to be algebraic if $\nu$ is the zero divisor of a polynomial. In this case, the counting
function $n_{\nu}$ is bounded (see \cite{GK1,WS1}).

\smallskip
With the help of the positive logarithm function, we define the proximity function of $f$ by
\[m(r, f)=\int_{\mathbb{C}^n\langle r\rangle} \log^+ |f|\;\sigma_n \geq  0.\]

The characteristic function of $f$ is defined by $T(r,f)=m(r,f)+N(r,f)$.
We define $m(r,a;f)=m(r,f)$ if $a=\infty$ and $m(r,a;f)=m(r,1/(f-a))$ if $a\in\mathbb{C}$. Now if $a\in\mathbb{C}$, then the first main theorem becomes $m(r,a;f)+N(r,a;f)=T(r,f) + O(1)$, where $O(1)$ denotes a bounded function when $r$ is sufficiently large.
We define the order of $f$ by
\[\rho(f):=\limsup _{r \rightarrow \infty} \frac{\log T(r, f)}{\log r}.\]

\smallskip
We define the linear measure $m(E):=\int_E dt$, the logarithmic measure $l(E):=\int_{E\cap [1,\infty)} \frac{d t}{t}$ and the upper density measure
\[\ol {\text{dens}}\;E=\lim\limits_{r\rightarrow \infty}\frac{1}{r}\int_{E\cap [1,r]} dt\]
for a set $E\subset [0,\infty)$. Moreover, if $l(E)<+\infty$, resp., $m(E)<+\infty$, then $E$ is of zero upper density.
Let $S(f)=\{g:\mathbb{C}^n\to\mathbb{P}^1\;\text{meromorphic}: \parallel T(r,g)=o(T(r,f))\}$, where $\mathbb{P}^1=\mathbb{C}\cup\{\infty\}$ and  $\parallel$ indicates that the equality holds only outside a set $E$ with zero upper density measure and the element in $S(f)$ is called the small function of $f$.

\section{{\bf Main results}}

Our goal is to investigate finite-order transcendental entire solutions $(f_1(z),f_2(z))$ of the system (\ref{pds1}) for the exponent regimes specified below. We determine explicit forms in the cases covered by the hypotheses and establish nonexistence in several complementary cases. Furthermore, we extend these findings to $\mathbb{C}^n$. We now state our main results.

\begin{theo}\label{t2.0} Let $c=(c_1,c_2,\ldots,c_n)\in\mathbb{C}^n$ and let $m_1,m_2,n_1,n_2\in\mathbb{N}$ such that $m_i+n_i>2$ for $i=1,2$. The finite-order entire solutions $(f_1,f_2)$ of {\em(\ref{pds1})} for which $\frac{\partial f_1}{\partial z_1}$ is constant are precisely
\begin{align*}
	f_1(z)=az_1+g_1(z_2,\ldots,z_n),\qquad
	f_2(z)=az_1+g_2(z_2,\ldots,z_n),
\end{align*}
where $a\in\mathbb{C}$ and $g_1,g_2$ are finite-order entire functions in $\mathbb{C}^{n-1}$ satisfying
\begin{align*}
&\left(ac_1+g_2(z_2+c_2,\ldots,z_n+c_n)-g_1(z_2,\ldots,z_n)\right)^{m_1}=1-a^{n_1},\\
&\left(ac_1+g_1(z_2+c_2,\ldots,z_n+c_n)-g_2(z_2,\ldots,z_n)\right)^{m_2}=1-a^{n_2}.
\end{align*}
\end{theo}

\begin{theo}\label{t2.1} Let $c=(c_1,c_2,\ldots,c_n)\in\mathbb{C}^n$ and let $m_1,m_2,n_1,n_2\in\mathbb{N}$ such that $m_i+n_i>2$ for $i=1,2$. If $n_i=m_i$ for $i=1,2$, then $n_i=m_i=2$ for $i=1,2$, and every finite-order entire solution $(f_1(z),f_2(z))$ of {\em(\ref{pds1})} such that $\frac{\partial f_1(z)}{\partial z_1}$ is non-constant has one of the following forms:
\begin{enumerate}
\item[\em{(1)}] $f_1(z)=z_1\cos P_1(z)+g_1(z_2,\ldots,z_n)$ and $f_2(z)=z_1\cos P_2(z)+g_2(z_2,\ldots,z_n)$, 
where $P_1(z)=A_{12}z_2+\ldots+A_{1n}z_n+Q(z_2,\ldots,z_n)$, $A_{1i}\in\mathbb{C}$ for $i=2,\ldots,n$ and $Q(z_2,\ldots,z_n)$ is a polynomial in $\mathbb{C}^{n-1}$ such that $Q(z_2+c_2,\ldots,z_n+c_n)=Q(z_2,\ldots,z_n)$, $P_1(z)+P_2(z)$ is a constant, and
\[\cos P_i(z+c)\equiv\cos P_j(z),\qquad i,j\in\{1,2\},\quad i\ne j.\]
Moreover, $g_1(z_2,\ldots,z_n)$ and $g_2(z_2,\ldots,z_n)$ are finite-order entire functions in $\mathbb{C}^{n-1}$ such that \begin{align}\label{T1}
	g_i(z_2+c_2,\ldots,z_n+c_n)-g_j(z_2,\ldots,z_n)=\sin P_j(z)-c_1\cos P_j(z),
\end{align}
where $i,j\in\{1,2\}$ such that $i\neq j$;
\item[\em{(2)}] $f_1(z)=-\frac{1}{2}\sin P(z)+g_1(z_2,\ldots,z_n)$ and $f_2(z)=\frac{1}{2}\sin (-P(z)+C)+ g_2(z_2,\ldots,z_n)$, 
where $P(z)=-2z_1+A_{12}z_2+\ldots+A_{1n}z_n+ Q(z_2,\ldots,z_n)$, $A_{1i},C\in\mathbb{C}$ for $i=2,\ldots,n$ and $Q(z_2,\ldots,z_n)$ is a polynomial in $\mathbb{C}^{n-1}$ such that
$Q(z_2+c_2,\ldots,z_n+c_n)=Q(z_2,\ldots,z_n)$ and  $e^{\iota(P(z)+C-P(z+c))}=-1, \; e^{\iota( -P(z)+C+P(z+c))}=-1$ and $e^{2\iota C }=1$,
$g_1(z_2,\ldots,z_n)$ and $g_2(z_2,\ldots,z_n)$ are finite-order entire functions in $\mathbb{C}^{n-1}$ such that $g_i(z_2+c_2,\ldots,z_n+c_n)\equiv  g_j(z_2,\ldots,z_n),\;\; i,j =1,2, i\ne j$;
\item[\em{(3)}] $f_1(z)=z_1\cos P_1(z)+g_1(z_2,\ldots,z_n)$ and $f_2(z)=z_1\cos P_2(z)+g_2(z_2,\ldots,z_n)$, 
where $P_1(z)=A_{12}z_2+\ldots+A_{1n}z_n+Q_1(z_2,\ldots,z_n)$ and $P_2(z)=B_{12}z_2+\ldots+B_{1n}z_n+Q_2(z_2,\ldots,z_n)$,
$A_{1i},B_{1i}\in\mathbb{C}$ for $i=2,\ldots,n$ and $Q_i(z_2,\ldots,z_n)$ is a polynomial in $\mathbb{C}^{n-1}$ such that $Q_i(z_2+2c_2,\ldots,z_n+2c_n)=Q_i(z_2,\ldots,z_n)$ for $i=1,2$, $P_1(z)-P_2(z+c)$ is a constant, and
\[\cos P_i(z+c)\equiv\cos P_j(z),\qquad i,j\in\{1,2\},\quad i\ne j.\]
Moreover, $g_1(z_2,\ldots,z_n)$ and $g_2(z_2,\ldots,z_n)$ are finite-order entire functions in $\mathbb{C}^{n-1}$ such that {\em(\ref{T1})} holds;

\item[\em{(4)}] $f_1(z)=-\frac{1}{2}\sin P_1(z)+g_1(z_2,\ldots,z_n)$ and $f_2(z)=-\frac{1}{2}\sin P_2(z)+ g_2(z_2,\ldots,z_n)$,
where $P_1(z)=-2z_1+A_{12}z_2+\ldots+A_{1n}z_n+ Q_1(z_2,\ldots,z_n)$ and $P_2(z)=-2z_1+B_{12}z_2+\ldots+B_{1n}z_n+ Q_2(z_2,\ldots,z_n)$, $A_{1i},B_{1i}\in\mathbb{C}$ for $i=2,\ldots,n$ and $Q_i(z_2,\ldots,z_n)$ is a polynomial in $\mathbb{C}^{n-1}$ such that
$Q_i(z_2+2c_2,\ldots,z_n+2c_n)=Q_i(z_2,\ldots,z_n)$, $i=1,2$, and
\[e^{\iota(P_i(z)-P_j(z+c))}=-1,\qquad i,j\in\{1,2\},\quad i\ne j.\]
The functions $g_1(z_2,\ldots,z_n)$ and $g_2(z_2,\ldots,z_n)$ are finite-order entire functions in $\mathbb{C}^{n-1}$ such that
\[g_i(z_2+c_2,\ldots,z_n+c_n)\equiv g_j(z_2,\ldots,z_n),\qquad i,j\in\{1,2\},\quad i\ne j.\]

\end{enumerate}
\end{theo}

\begin{table}[H]
	\centering
	\setlength{\arrayrulewidth}{1.0pt}
	\renewcommand{\arraystretch}{1.12}
	\setlength{\tabcolsep}{2pt}
	\footnotesize
	
	\begin{tabular}{|c|c|c|}
		\hline
		
		\rowcolor{blue!25}
		\textbf{Case} &
		\textbf{Representation} &
		\textbf{Essential conditions}
		\\
		\hline

		
		\rowcolor{green!10}
		(1)
		&
		\begin{minipage}[c]{3.9cm}
			\centering
			$f_1=z_1\cos P_1+g_1$
			
			$f_2=z_1\cos P_2+g_2$
		\end{minipage}
		&
		\begin{minipage}[c]{9.1cm}
			\centering
			$n_i=m_i=2\;(i=1,2)$;
			
			$P_1=A_{12}z_2+\cdots+A_{1n}z_n+Q$,
			
			$Q(z+c')=Q(z')$, \quad $P_1+P_2=\text{constant}$;
			
			$\cos P_i(z+c)=\cos P_j(z)$ for $i\ne j$;
			
			$g_1,g_2$ satisfy \eqref{T1}.
		\end{minipage}
		\\
		\hline

		
		\rowcolor{yellow!15}
		(2)
		&
		\begin{minipage}[c]{3.9cm}
			\centering
			$f_1=-\frac12\sin P+g_1$
			
			$f_2=\frac12\sin(-P+C)+g_2$
		\end{minipage}
		&
		\begin{minipage}[c]{9.1cm}
			\centering
			$P=-2z_1+A_{12}z_2+\cdots+A_{1n}z_n+Q$;
			
			$Q(z+c')=Q(z')$;
			
			$e^{\iota(P+C-P(z+c))}=-1$,
			$e^{\iota(-P+C+P(z+c))}=-1$,
			$e^{2\iota C}=1$;
			
			$g_i(z+c')\equiv g_j(z')$.
		\end{minipage}
		\\
		\hline

		
		\rowcolor{orange!12}
		(3)
		&
		\begin{minipage}[c]{3.9cm}
			\centering
			$f_1=z_1\cos P_1+g_1$
			
			$f_2=z_1\cos P_2+g_2$
		\end{minipage}
		&
		\begin{minipage}[c]{9.1cm}
			\centering
			$P_1=A_{12}z_2+\cdots+A_{1n}z_n+Q_1$,
			
			$P_2=B_{12}z_2+\cdots+B_{1n}z_n+Q_2$;
			
			$Q_i(z+2c')=Q_i(z')$;
			
			$P_1(z)-P_2(z+c)=\text{constant}$;
			
			$\cos P_i(z+c)=\cos P_j(z)$ for $i\ne j$;
			
			$g_1,g_2$ satisfy \eqref{T1}.
		\end{minipage}
		\\
		\hline

		
		\rowcolor{red!12}
		(4)
		&
		\begin{minipage}[c]{3.9cm}
			\centering
			$f_1=-\frac12\sin P_1+g_1$
			
			$f_2=-\frac12\sin P_2+g_2$
		\end{minipage}
		&
		\begin{minipage}[c]{9.1cm}
			\centering
			$P_1=-2z_1+A_{12}z_2+\cdots+A_{1n}z_n+Q_1$,
			
			$P_2=-2z_1+B_{12}z_2+\cdots+B_{1n}z_n+Q_2$;
			
			$Q_i(z+2c')=Q_i(z')$;
			
			$e^{\iota(P_i(z)-P_j(z+c))}=-1$ for $i\ne j$;
			
			$g_i(z+c')\equiv g_j(z')$ for $i\ne j$.
		\end{minipage}
		\\
		\hline
		
	\end{tabular}
	
	\caption{Solution families in Theorem~\ref{t2.1}.}
	\label{tab:classification}
\end{table}

The following example explains the conclusion of Theorem \ref{t2.0}.
\begin{exm} Suppose that $(f_1(z),f_2(z))=(z_1+z_2,z_1+z_2)$ and $c=(c_1,c_2,\ldots,c_n)$ such that $c_1+c_2=0$. Here, $g_i(z_2,z_3,\ldots,z_n)=z_2$. Clearly, $\frac{\partial f_1}{\partial z_1}$ is constant and $(f_1(z),f_2(z))$ is a solution of
\begin{equation}\label{PDDE1}
\begin{cases}
(\partial_{z_1}f_1(z))^{2}+(f_2(z+c)-f_1(z))^{2}=1,\\[4pt]
(\partial_{z_1}f_2(z))^{2}+(f_1(z+c)-f_2(z))^{2}=1.
\end{cases}
\end{equation}
\end{exm}
The next example shows the actual existence of the conclusion $(1)$ of Theorem \ref{t2.1}.
\begin{exm}
Let
\[
f_1(z)=z_1\cos z_2,\qquad f_2(z)=z_1\cos z_2+\sin z_2,
\]
and let $c=(0,0,c_3,\ldots,c_n)$. Take $P_1(z)=z_2$, $P_2(z)=-z_2$, $g_1\equiv0$, and $g_2(z_2,\ldots,z_n)=\sin z_2$. Then $P_1+P_2=0$, the two relations in (\ref{T1}) hold, and
\[
\frac{\partial f_1}{\partial z_1}=\frac{\partial f_2}{\partial z_1}=\cos z_2,
\qquad f_2(z+c)-f_1(z)=\sin z_2,
\qquad f_1(z+c)-f_2(z)=-\sin z_2.
\]
Hence the pair satisfies {\em(\ref{PDDE1})} and has non-constant $z_1$-derivatives.
\end{exm}
The next example shows the actual existence of the conclusion $(2)$ of Theorem \ref{t2.1}.
\begin{exm}
Let
\[(f_1(z),f_2(z))=\left(\frac{1}{2}\sin(2z_1),\frac{1}{2}\sin(2z_1)\right).\]
Here $P(z)=-2z_1$, $C=0$, $Q\equiv g_i\equiv0$, and
$c=\left(-\frac{\pi}{2},c_2,c_3,\ldots,c_n\right)$. Then
$P(z+c)=P(z)+\pi$, so both exponential conditions in conclusion $(2)$ equal $-1$. Direct substitution gives
\[
\left(\frac{\partial f_i}{\partial z_1}\right)^2+\left(f_j(z+c)-f_i(z)\right)^2
=\cos^2(2z_1)+\sin^2(2z_1)=1,
\]
for $i,j\in\{1,2\}$ with $i\ne j$.
\end{exm}
The next example shows the actual existence of the conclusion $(3)$ of Theorem \ref{t2.1}.
\begin{exm}	Take
\[
f_1(z)=f_2(z)=z_1\cos z_2+\frac{z_2}{2\pi}\sin z_2
\]
and $c=(0,2\pi,0,\ldots,0)$. Here $P_1(z)=P_2(z)=z_2$, $Q_i\equiv0$, and
\[
g(z_2+2\pi)-g(z_2)=\sin z_2,
\qquad g(z_2)=\frac{z_2}{2\pi}\sin z_2.
\]
Consequently $\partial f_i/\partial z_1=\cos z_2$ and
$f_j(z+c)-f_i(z)=\sin z_2$ for $i\ne j$, so the pair satisfies {\em(\ref{PDDE1})}.
\end{exm}
The following example shows the actual existence of the conclusion $(4)$ of Theorem \ref{t2.1}.
\begin{exm}
Suppose $(f_1(z),f_2(z))=\left(\frac{1}{2}\sin (2z_1),\frac{1}{2}\sin (2z_1)\right).$ Here $P_i(z)=-2z_1$, $A_{1i}=B_{1i}=0$ for all $i=1,2,\ldots,n$, $g_i=Q_i\equiv 0$, $c=\left(\frac{\pi}{2},c_2,c_3,\ldots,c_n\right)$. Clearly, $(f_1(z),f_2(z))$ is a solution of {\em(\ref{PDDE1})}.
\end{exm}

\begin{theo}\label{t2.2} Let $c\in\mathbb{C}^n$ and let $m_1,m_2,n_1,n_2\in\mathbb{N}$ such that $m_i+n_i>2$ for $i=1,2$ and $n_i\neq m_i$ for at least one $i\in\{1,2\}$. For finite-order transcendental entire solutions $(f_1(z),f_2(z))$ of {\em(\ref{pds1})} such that $\frac{\partial f_1(z)}{\partial z_1}$ is non-constant, the following conclusions hold in the indicated exponent regimes:
\begin{enumerate}
\item[\em{(1)}] Suppose $n_1=m_1$. Then $n_1=m_1=2$ and
\begin{enumerate}
\item[\em{(1A)}] if $n_2>m_2$, then the system of equations {\em(\ref{pds1})} has no finite-order transcendental entire solutions;
\item[\em{(1B)}] if $n_2<m_2$, then $n_2=1$ and $m_2=2$ and one of the following cases holds
\begin{enumerate}
\item[\em{(1B) (i)}] $f_1(z)=z_1\cos h(z)+\phi(z_2,\ldots,z_n)$ and $f_2(z)=(z_1-c_1)\cos h(z-c)+\sin h(z-c)+\phi(z_2-c_2,\ldots,z_n-c_n)$, where $\phi(z_2,\ldots,z_n)$ is an entire function in $\mathbb{C}^{n-1}$ and $h(z)$ is a polynomial in $\mathbb{C}^n$ such that $\frac{\partial h(z)}{\partial z_1}\equiv 0$ and 
\begin{align*}
&(z_1+2c_1)\cos h(z+2c)+\phi(z_2+2c_2,\ldots,z_n+2c_n)-
z_1\cos h(z)-\sin h(z)\\&-\phi(z_2,\ldots,z_n)=\pm \sqrt{2}\sin \frac{1}{2}h(z).
\end{align*}
\item[\em{(1B) (ii)}] $f_1(z)=-\frac{1}{2}\sin h(z)+\psi(z_2,\ldots,z_n)$ and $f_2(z)=\frac{1}{2}\sin h(z-c)+\psi(z_2-c_2,\ldots,z_n-c_n)$, where $\psi(z_2,\ldots,z_n)$ is an entire function in $\mathbb{C}^{n-1}$ and $h(z)$ is a polynomial in $\mathbb{C}^n$ such that $\frac{\partial h(z)}{\partial z_1}\equiv -2$ and
\begin{align*}
&-\frac{1}{2}\sin h(z+2c)+\psi(z_2+2c_2,\ldots,z_n+2c_n)
-\frac{1}{2}\sin h(z)-\psi(z_2,\ldots,z_n)\\&=\pm \sqrt{2}\cos \frac{1}{2}h(z);
\end{align*}
\end{enumerate}
\end{enumerate}

\item[\em{(2)}] Suppose $m_1>n_1$ and $f_j(z+c)-f_i(z)$, $f_j(z)$ share $0$ CM and $\rho((f_j(z+c)-f_i(z))/f_j(z))<1$, where $i,j\in\{1,2\}$ such that $i\neq j$. 
\begin{enumerate}
\item[\em{(2A)}] If $n_2>m_2$ such that $n_2>2$ and $m_1>\frac{n_2}{n_2-2}$, then the system of equations {\em(\ref{pds1})} does not have any finite-order transcendental entire solutions;
\item[\em{(2B)}] If $n_2<m_2$, then the system of equations {\em(\ref{pds1})} does not have any finite-order transcendental entire solutions;
\item[\em{(2C)}] If $n_2=m_2$, then $n_2=m_2=2$ and the conclusions (1B)(i) and (1B)(ii) hold;
\end{enumerate}
\item[\em{(3)}] Suppose $m_1<n_1$ and $f_j(z+c)-f_i(z)$, $f_j(z)$ share $0$ CM and $\rho((f_j(z+c)-f_i(z))/f_j(z))<1$, where $i,j\in\{1,2\}$ such that $i\neq j$. 
\begin{enumerate}
\item[\em{(3A)}] If $n_2<m_2$ such that $n_1>2$ and $m_2>\frac{n_1}{n_1-2}$, then the system of equations {\em (\ref{pds1})} does not have any finite-order transcendental entire solutions;
\item[\em{(3B)}] If $n_2>m_2$, then the system of equations {\em(\ref{pds1})} does not have any finite-order transcendental entire solutions;
\item[\em{(3C)}] If $n_2=m_2$, then the system of equations {\em(\ref{pds1})} does not have any finite-order transcendental entire solutions.
\end{enumerate}

\end{enumerate}
\end{theo}

\begin{rem}
Parts \emph{(2A)} and \emph{(3A)} apply only under their displayed strict inequalities. No conclusion is asserted here for the remaining exponent configurations not covered by those inequalities.
\end{rem}

\newcommand{\shareCM}{%
	$f_j(z+c)-f_i(z),\ f_j(z)$ share $0$ CM, and\newline
	$\displaystyle \rho\!\left(\frac{f_j(z+c)-f_i(z)}{f_j(z)}\right)<1,$\newline
	$i,j\in\{1,2\}$, $i\ne j$.%
}

\begin{center}
	\setlength{\arrayrulewidth}{0.9pt}
	\setlength{\tabcolsep}{4pt}
	\renewcommand{\arraystretch}{1.2}
	\footnotesize
	\begin{longtable}{|>{\centering\arraybackslash}p{1.3cm}|
			>{\centering\arraybackslash}p{4.3cm}|
			>{\centering\arraybackslash}p{8.6cm}|}
		
		\caption{Summary of the conclusions in Theorem~\ref{t2.2}.}
		\label{tab:classification-t22}\\
		\hline
		\rowcolor{blue!25}
		\textbf{Case} & \textbf{Representation} & \textbf{Essential conditions} \\
		\hline
		\endfirsthead
		
		\multicolumn{3}{c}{\tablename\ \thetable{} -- continued from previous page}\\
		\hline
		\rowcolor{blue!25}
		\textbf{Case} & \textbf{Representation} & \textbf{Essential conditions} \\
		\hline
		\endhead
		
		\hline
		\multicolumn{3}{r}{\textit{Continued on next page}}\\
		\endfoot
		
		\hline
		\endlastfoot

		
		\multicolumn{3}{|c|}{\cellcolor{cyan!25}\textbf{Case when $n_1=m_1$}} \\
		\hline
		
		\rowcolor{cyan!10}
		(1A) & No finite-order transcendental entire solutions. &
		$n_1=m_1=2,\quad n_2>m_2.$ \\
		\hline
		
		\rowcolor{yellow!18}
		(1B)(i) &
		$f_1=z_1\cos h+\phi,$\newline
		$f_2=(z_1-c_1)\cos h(z-c)+\sin h(z-c)+\phi(z-c').$
		&
		$n_1=m_1=2,\ n_2=1,\ m_2=2$;\newline
		$h$ is a polynomial in $\mathbb{C}^n$ satisfying $\dfrac{\partial h}{\partial z_1}\equiv0$;\newline
		$\phi(z')$ is entire in $\mathbb{C}^{n-1}$;\newline
		$(z_1+2c_1)\cos h(z+2c)+\phi(z'+2c')-z_1\cos h(z)-\sin h(z)-\phi(z')$\newline
		$\hspace{1cm}=\pm\sqrt{2}\sin\frac{1}{2}h(z).$ \\
		\hline
		
		\rowcolor{yellow!18}
		(1B)(ii) &
		$f_1=-\frac12\sin h+\psi,$\newline
		$f_2=\frac12\sin h(z-c)+\psi(z-c').$
		&
		$n_1=m_1=2,\ n_2=1,\ m_2=2$;\newline
		$h$ is a polynomial in $\mathbb{C}^n$ satisfying $\dfrac{\partial h}{\partial z_1}\equiv-2$;\newline
		$\psi(z')$ is entire in $\mathbb{C}^{n-1}$;\newline
		$-\frac12\sin h(z+2c)+\psi(z'+2c')-\frac12\sin h(z)-\psi(z')$\newline
		$\hspace{1cm}=\pm\sqrt{2}\cos\frac12h(z).$ \\
		\hline

		
		\multicolumn{3}{|c|}{\cellcolor{green!22}\textbf{Case when $m_1>n_1$}} \\
		\hline
		
		\rowcolor{green!10}
		(2A) & No finite-order transcendental entire solutions. &
		$n_2>m_2,\quad n_2>2,\quad m_1>\dfrac{n_2}{n_2-2}.$\newline
		\shareCM \\
		\hline
		
		\rowcolor{green!10}
		(2B) & No finite-order transcendental entire solutions. &
		$n_2<m_2.$\newline
		\shareCM \\
		\hline
		
		\rowcolor{green!10}
		(2C) & Conclusions (1B)(i) and (1B)(ii) hold. &
		$n_2=m_2=2.$\newline
		\shareCM \\
		\hline

		
		\multicolumn{3}{|c|}{\cellcolor{orange!22}\textbf{Case when $m_1<n_1$}} \\
		\hline
		
		\rowcolor{orange!10}
		(3A) & No finite-order transcendental entire solutions. &
		$n_2<m_2,\quad n_1>2,\quad m_2>\dfrac{n_1}{n_1-2}.$\newline
		\shareCM \\
		\hline
		
		\rowcolor{orange!10}
		(3B) & No finite-order transcendental entire solutions. &
		$n_2>m_2.$\newline
		\shareCM \\
		\hline
		
		\rowcolor{orange!10}
		(3C) & No finite-order transcendental entire solutions. &
		$n_2=m_2.$\newline
		\shareCM \\
		\hline
		
	\end{longtable}
\end{center}

\section{\bf{Key lemmas}}
\begin{lem}\label{L.1} \cite[Lemma 1.37]{HLY1} Let $f$ be a non-constant meromorphic function in $\mathbb{C}^n$ and $I=(\alpha_1,\ldots,\alpha_n)\in \mathbb{Z}^n_+$ be a multi-index. Then for any $\varepsilon>0$, we have
\[m\left(r,\partial^I(f)/f\right)\leq |I|\log^+T(r,f)+|I|(1+\varepsilon)\log^+\log T(r,f)+O(1)\]
for all $r\not\in E$, where $l(E)<+\infty$.
\end{lem}

\begin{lem}\label{L.2} \cite[Lemma 1.2]{HY1} Let $f$ be a non-constant meromorphic function in $\mathbb{C}^n$ and let $a_1,a_2,\ldots,a_q$ be different points in $\mathbb{C}\cup \{\infty\}$. Then
\beas \parallel (q-2)T(r,f)\leq \sideset{}{_{j=1}^{q}}{\sum} \ol N(r,a_j;f)+O(\log (rT(r,f))).\eeas
\end{lem}

\begin{lem}\label{L.3} \cite[Theorem 1.26]{HLY1} Let $f$ be non-constant meromorphic function in $\mathbb{C}^n$. Assume that
$R(z, w)=\frac{A(z, w)}{B(z, w)}$. Then
\beas T\left(r, R_f\right)=\max \{p, q\} T(r, f)+O\Big(\sideset{}{_{j=0}^p}{\sum} T(r, a_j)+\sideset{}{_{j=0}^q}{\sum}T(r, b_j)\Big),\eeas
where $R_f(z)=R(z, f(z))$ and two coprime polynomials $A(z, w)$ and $B(z,w)$ are given
respectively as follows:
\[A(z,w)=\sideset{}{_{j=0}^p}{\sum} a_j(z)w^j\;\;\text{and}\;\;B(z,w)=\sideset{}{_{j=0}^q}{\sum} b_j(z)w^j.\]
\end{lem}

\begin{lem}\label{L.6}\cite[Lemma 3.2]{HLY1} Let $f_j\not\equiv 0\;(j=1,2,\ldots,n;n\geq 3)$ be meromorphic functions in $\mathbb{C}^n$ such that $f_1,\ldots,f_{n-1}$ are non-constants and $f_1+\cdots+f_n=1$. If
\beas \parallel \;\;\sideset{}{_{j=1}^n}{\sum}\Big\lbrace N_{n-1}(r,0;f_j)+(n-1)\ol{N}(r,f_j)\Big\rbrace<\lambda T(r,f_k)+O(\log^+(T(r,f_k)))\eeas
holds for $k=1,2,\ldots,n-1$, where $\lambda<1$, then $f_n\equiv 1$.
\end{lem}

\begin{lem}\label{L.5}\cite[Proposition 3.2]{hy1} Let $P$ be a non-constant entire function in $\mathbb{C}^n$. Then
\[\rho(e^P)=
\begin{cases}
\deg(P) & \text{if $P$ is a polynomial,}\\
+\infty & \text{otherwise}.
\end{cases}\]
\end{lem}

\begin{lem}\label{L.8}\cite[Theorem 2.1]{CX1} Let $f$ be a non-constant meromorphic function in $\mathbb{C}^n$ and let $c\in \mathbb{C}^n$. If
\bea\label{Ss}\limsup\limits_{r\rightarrow \infty} \frac{\log T(r,f)}{r}=0,\eea
then
\beas \parallel\;m\left(r,f(z+c)/f(z)\right)+m\left(r,f(z)/f(z+c)\right)=o(T(r,f)).\eeas
\end{lem}

\begin{lem}\label{L.7}\cite[Theorem 2.2]{CX1} Let $f$ be a non-constant meromorphic function on $\mathbb{C}^n$ such that {\em(\ref{Ss})} holds,
then $\parallel\;T(r,f(z+c))=T(r,f)+o(T(r,f))$
holds for any constant $c\in\mathbb{C}^n$.
\end{lem}

Let $f$ be a non-constant meromorphic function in $\mathbb{C}^n$. Define complex differential-difference polynomials as follows:
\bea\label{cl1} P(f(z))=\sum\limits_{\mathbf{p} \in I} a_{\mathbf{p}}(z) f^{p_0}(z)\big(\partial^{\mathbf{i}_1}(f(z))\big)^{p_1} \cdots\big(\partial^{\mathbf{i}_l}(f(z))\big)^{p_l}f^{p_{l+1}}(z + \hat q_{l+1}) \cdots f^{p_{l+s}}(z + \hat q_{l+s}),\eea
$\mathbf{p}=\left(p_0, \ldots, p_{l+s}\right) \in \mathbb{Z}_{+}^{l+s+1}$,
\bea\label{cl2} Q(f(z))=\sum\limits_{\mathbf{q} \in J} c_{\mathbf{q}}(z) f^{q_0}(z)\big(\partial^{\mathbf{j}_1}(f(z))\big)^{q_1} \cdots\big(\partial^{\mathbf{j}_l}(f(z))\big)^{q_l}f^{q_{l+1}}(z+\tilde q_{l+1}) \cdots f^{q_{l+t}}(z+\tilde q_{l+t}),\eea
$\mathbf{q}=\left(q_0, \ldots, q_{l+t}\right) \in \mathbb{Z}_{+}^{l+t+1}$ and
\bea\label{cl3} B(f(z))=\sideset{}{_{k=0}^{m}}{\sum} b_k(z) f^k(z),\eea
where $I, J$ are finite sets of distinct elements, $a_{\mathbf{p}}, b_k, c_{\mathbf{q}}\in S(f)$ such that $b_n\not\equiv 0$ and $\hat q_i, \tilde q_j\in\mathbb{C}^n$.

For the case, when $P(f(z))$ and $Q(f(z))$ are complex differential polynomials, Hu and Yang \cite[Lemma 2.1]{ps1} generalised Clunie-lemma to high dimension. In 2020, Cao and Xu \cite[Theorem 3.6]{CX1} improved and extended Laine-Yang's difference analogue of Clunie theorem in one variable \cite[Theorem 2.3]{LY1} to high dimension by using Lemma \ref{L.8}. Now by using Lemmas \ref{L.1} and \ref{L.8} and proceeding in the same way as done in the proofs of Lemma 2.1 \cite{ps1} and Theorem 3.6 \cite{CX1}, we get the following lemma.

\begin{lem}\label{L.4}\cite[Theorem 3.6]{CX1}; see also \cite[Lemma 2.1]{ps1}. Let $f$ be a non-constant meromorphic function in $\mathbb{C}^n$ such that {\em(\ref{Ss})} holds.
Suppose $f$ satisfies the complex differential-difference equation $B(f) Q(f)=P(f)$,
where $P(f)$, $Q(f)$ and $B(f)$ are defined as in {\em(\ref{cl1})}, {\em(\ref{cl2})} and {\em(\ref{cl3})} respectively.
If $\deg(P(f)) \leq m=\deg(B(f))$, then $\parallel\;m(r, Q(f))=o(T(r, f))$.
\end{lem}

\begin{lem}\label{L.8a}\cite[Lemma 1.68]{HLY1} Let $f_1:\mathbb{C}^n\to\mathbb{P}^1$ and $f_2:\mathbb{C}^n\to\mathbb{P}^1$ be two non-constant meromorphic functions. Then for $r>0$ we have
\[N(r,0;f_1f_2)-N(r,f_1f_2)=N(r,0;f_1)+N(r,0;f_2)-N(r,f_1)-N(r,f_2).\]
\end{lem}

\begin{lem}\label{L.9} Let $m_1,m_2,n_1,n_2$ be positive integers such that $n_i+m_i>2$ for $i=1,2$. If $(f_1(z),f_2(z))$ is a pair of finite-order entire solution of the system of equations {\em(\ref{pds1})} in $\mathbb{C}^n$ such that $\frac{\partial f_1(z)}{\partial z_1}$ is non-constant, then 
\begin{align}\label{lm.5}
\frac{1}{n_i}+\frac{1}{m_i}\geq 1,
\end{align}
where $i\in\{1,2\}$.
\end{lem}
\begin{proof}
Let $(f_1(z),f_2(z))$ be a pair of finite-order non-constant entire solutions of (\ref{pds1}) in $\mathbb{C}^n$. We first verify that $\frac{\partial f_2}{\partial z_1}$ is non-constant. If $\frac{\partial f_2}{\partial z_1}=b$ were constant, then the second equation of (\ref{pds1}) would imply that
\[
\left(f_1(z+c)-f_2(z)\right)^{m_2}=1-b^{n_2}.
\]
Hence $f_1(z+c)-f_2(z)$ would be constant. Differentiating with respect to $z_1$ would give
$\frac{\partial f_1(z+c)}{\partial z_1}=b$, contradicting the assumed non-constancy of $\frac{\partial f_1}{\partial z_1}$. Set
\begin{align}\label{lm} 
F_i(z)=\frac{\partial f_i(z)}{\partial z_1}\;\;\text{and}\;\;G_i(z)=f_j(z+c)-f_i(z),
\end{align}
 where $i,j\in\{1,2\}$ such that $i\neq j$. Let
\begin{align}\label{lm.3} 
h_i(z)=\frac{F_i^{n_i}(z)-1}{F_i^{n_i}(z)},
\end{align}
where $i\in\{1,2\}$. Clearly $h_i(z)$ is a non-constant meromorphic function in $\mathbb{C}^n$. Using Lemma \ref{L.3} to (\ref{lm.3}), we get
\begin{align}\label{lm.4}
\parallel T(r,h_i)+o(T(r,h_i))=n_iT(r,F_i)+o(T(r,F_i)),
\end{align}
where $i\in\{1,2\}$. Now using (\ref{pds1}) to (\ref{lm.3}), one can easily deduce that
\begin{align*}
\parallel\;\ol N(r,h_i)\leq \ol N(r,0, F_i^{n_i})=\ol N(r,0,F_i)+o(T(r,F_i)),
\end{align*}
\begin{align*}
\parallel\;\ol N(r,0,h_i)=\ol N(r,1,F_i^{n_i})\leq \ol N(r,0,G_i^{m_i})=\ol N(r,0,G_i)+o(T(r,G_i))
\end{align*}
and $\parallel \ol N(r,1,h_i)=0$,
where $i,j\in\{1,2\}$. Therefore in view of the second main theorem and using Lemma \ref{L.2} and (\ref{lm.4}), we get
\begin{align}\label{lm.6} 
\parallel\;n_i T(r,F_i)&=T(r,h_i)+o(T(r,h_i))\\&\leq \ol N(r,h_i)+\ol N(r,0,h_i)+\ol N(r,1,h_i)+o(T(r,h_i))\nonumber\\&\leq
\ol N(r,0,F_i)+\ol N(r,0,G_i)+o(T(r,F_i))+o(T(r,G_i))\nonumber\\&\leq
T(r,F_i)+T(r,G_i)+o(T(r,F_i))+o(T(r,G_i)),\nonumber
\end{align}
where $i,j\in\{1,2\}$. Also using Lemma \ref{L.3} to the system of equations (\ref{pds1}), we get
\begin{align}\label{lm.0} 
\parallel\;n_i T(r,F_i)+o(T(r,F_i))=m_i T(r,G_i)+o(T(r,G_i)),
\end{align}
where $i\in\{1,2\}$. Using (\ref{lm.0}) to (\ref{lm.6}), we get
\begin{align*}
\parallel\;\left(n_i-1-\frac{n_i}{m_i}\right)T(r,F_i)\leq o(T(r,F_i)),
\end{align*}
where $i\in\{1,2\}$ and so for $i\in\{1,2\}$, we have
\begin{align*}
\frac{1}{n_i}+\frac{1}{m_i}\geq 1.
\end{align*}
\end{proof}

\section {{\bf Proof of Theorem \ref{t2.0}}}
\begin{proof}Let $(f_1(z),f_2(z))$ be a pair of finite order entire solutions to the system of equations (\ref{pds1}) in $\mathbb{C}^n$ such that $\frac{\partial f_1(z)}{\partial z_1}$ is constant. We know that $n_i+m_i>2$ for $i=1,2$. Consequently from (\ref{pds1}), we have respectively
	\begin{align}\label{T0.1} 
		\left(\frac{\partial f_1(z)}{\partial z_1}\right)^{n_1}+\left(f_2(z+c)-f_1(z)\right)^{m_1}=1
	\end{align}
	and 
	\begin{align}\label{T0.2}
		\left(\frac{\partial f_2(z)}{\partial z_1}\right)^{n_2}+\left(f_1(z+c)-f_2(z)\right)^{m_2}=1.
	\end{align}

Suppose $\frac{\partial f_1(z)}{\partial z_1}=a$. Then \eqref{T0.1} yields
\begin{align}\label{T0.3}
\left(f_2(z+c)-f_1(z)\right)^{m_1}=1-a^{n_1}.
\end{align}	
	
Thus the entire function $f_2(z+c)-f_1(z)$ is constant. Differentiating it with respect to $z_1$ gives
\[
\frac{\partial f_2(z+c)}{\partial z_1}-a=0,
\]
and hence $\frac{\partial f_2(z)}{\partial z_1}=a$. On integration, we get
\begin{align}\label{T0.4}
	f_1(z)=az_1+g_1(z_2,\ldots,z_n)\; \text{and}\; f_2(z)=az_1+g_2(z_2,\ldots,z_n), 
\end{align}
where $g_1(z_2,\ldots,z_n)$ and $g_2(z_2,\ldots,z_n)$ are finite order entire functions in $\mathbb{C}^{n-1}$.
Now using (\ref{T0.4}) to (\ref{T0.1}) and (\ref{T0.2}), we get respectively
\begin{align*}
\left(ac_1+g_2(z_2+c_2,\ldots,z_n+c_n)-g_1(z_2,\ldots,z_n)\right)^{m_1}=1-a^{n_1}
\end{align*}
and
\begin{align*}
\left(ac_1+g_1(z_2+c_2,\ldots,z_n+c_n)-g_2(z_2,\ldots,z_n)\right)^{m_2}=1-a^{n_2}.
\end{align*}

Conversely, direct substitution shows that every pair in the statement satisfying the two displayed relations solves (\ref{pds1}).

\end{proof}

\section {{\bf Proof of Theorem \ref{t2.1}}}
\begin{proof}Let $(f_1(z),f_2(z))$ be a pair of finite-order entire solutions to the system of equations (\ref{pds1}) in $\mathbb{C}^n$ such that $\frac{\partial f_1(z)}{\partial z_1}$ is non-constant. The argument at the beginning of the proof of Lemma~\ref{L.9} shows that $\frac{\partial f_2(z)}{\partial z_1}$ is also non-constant.
We have $n_i=m_i$ for $i=1,2$. Since $n_i+m_i>2$, from (\ref{lm.5}), we get $n_i=m_i=2$ for $i=1,2$. Consequently from (\ref{pds1}), we have respectively
\begin{align}\label{T1.1} 
\left(\frac{\partial f_1(z)}{\partial z_1}\right)^2+\left(f_2(z+c)-f_1(z)\right)^2=1
\end{align}
and 
\begin{align}\label{T1.2}
\left(\frac{\partial f_2(z)}{\partial z_1}\right)^2+\left(f_1(z+c)-f_2(z)\right)^2=1.
\end{align}

Applying Theorem D to (\ref{T1.1}), we have
\begin{align}\label{T1.3} 
\frac{\partial f_1(z)}{\partial z_1}=\cos P_1(z)
\end{align}
and
\begin{align}\label{T1.4} 
f_2(z+c)-f_1(z)=\sin P_1(z),
\end{align}
where $P_1(z)$ is an entire function in $\mathbb{C}^n$.	Again applying Theorem D to (\ref{T1.2}), we obtain
\begin{align}\label{T1.5} 
\frac{\partial f_2(z)}{\partial z_1}=\cos P_2(z)
\end{align}
and
\begin{align}\label{T1.6} 
f_1(z+c)-f_2(z)=\sin P_2(z),
\end{align}
where $P_2(z)$ is an entire function in $\mathbb{C}^n$.
As $\rho(f_1)<+\infty$, using Lemma \ref{L.1}, we get
\begin{align*}
T\left(r,\frac{\partial f_1(z)}{\partial z_1}\right)=m\left(r, \frac{\partial f_1(z)}{\partial z_1}\right)\leq m(r,f_1)+O(\log r)\leq T(r,f_1)+O(\log r),
\end{align*}
which shows that $\rho\left(\frac{\partial f_1(z)}{\partial z_1}\right)<+\infty$.  Now using Lemmas \ref{L.3} and \ref{L.5} to (\ref{T1.3}), we conclude that $P_1(z)$ is a polynomial in $\mathbb{C}^n$. Similarly we can prove that $P_2(z)$ is a polynomial in $\mathbb{C}^n$.

Differentiating (\ref{T1.4}) with respect to $z_1$ and then using (\ref{T1.3}) and (\ref{T1.5}), we have
\begin{align}\label{T1.7} 
\cos P_2(z+c)-\cos P_1(z)=\frac{\partial P_1(z)}{\partial z_1}\cos P_1(z),
\end{align}
which implies that $\frac{\partial P_1(z)}{\partial z_1}\not\equiv-1$. 
Now, from (\ref{T1.7}), we get 
\begin{align}\label{T1.8} 
\left(\frac{\partial P_1(z)}{\partial z_1}+1\right)e^{\iota (P_1(z)+P_2(z+c))} +\left(\frac{\partial P_1(z)}{\partial z_1}+1\right)e^{-\iota (P_1(z)-P_2(z+c))}-e^{2\iota P_2(z+c)}=1.
\end{align}

Similarly differentiating (\ref{T1.6}) with respect to $z_1$ and then using (\ref{T1.3}) and (\ref{T1.5}), we get
\begin{align}\label{T1.9} 
\left(\frac{\partial P_2(z)}{\partial z_1}+1\right)e^{\iota (P_2(z)+P_1(z+c))} +\left(\frac{\partial P_2(z)}{\partial z_1}+1\right)e^{-\iota (P_2(z)-P_1(z+c))}-e^{2\iota P_1(z+c)}=1,
\end{align}
where 
\begin{align}\label{T1.10} 
\cos P_1(z+c)-\cos P_2(z)=\frac{\partial P_2(z)}{\partial z_1}\cos P_2(z)
\end{align}
and 
$\frac{\partial P_2(z)}{\partial z_1}\not\equiv -1$. Now using Lemma \ref{L.6} to (\ref{T1.8}) and (\ref{T1.9}), we have respectively
\begin{align}\label{T1.11}
 \text{either}\;\;\left(\frac{\partial P_1(z)}{\partial z_1}+1\right)e^{\iota (P_1(z)+P_2(z+c))}=1 \;\;\text{or}\;\;\left( \frac{\partial P_1(z)}{\partial z_1}+1\right)e^{-\iota (P_1(z)-P_2(z+c))}=1
 \end{align}
and
\begin{align}\label{T1.12}
 \text{either}\;\;\left(\frac{\partial P_2(z)}{\partial z_1}+1\right)e^{\iota (P_2(z)+P_1(z+c))}=1 \;\;\text{or}\;\; \left(\frac{\partial P_2(z)}{\partial z_1}+1\right)e^{-\iota (P_2(z)-P_1(z+c))}=1.
 \end{align}

We know that $\frac{\partial f_1(z)}{\partial z_1}$ is non-constant. Then from (\ref{T1.3}), we see that $P_1(z)$ is non-constant. Consequently from (\ref{T1.11}), we conclude that $P_2(z)$ is also non-constant. We now consider the following four cases.

\medskip
{\bf Case 1.}
Let
\begin{align}\label{T1.15} 
\begin{cases}  \left(\frac{\partial P_1(z)}{\partial z_1}+1\right)e^{\iota (P_1(z)+P_2(z+c))}=1\\ \left(\frac{\partial P_2(z)}{\partial z_1}+1\right)e^{\iota (P_2(z)+P_1(z+c))}=1.\end{cases}
\end{align}
	 
It follows from (\ref{T1.15}) that $P_1(z+c)+P_2(z)$ and $P_2(z+c)+P_1(z)$ are both constants. Consequently $P_1(z+c)+P_2(z+c)+P_1(z)+P_2(z)$ is a constant and so $P_1(z)+P_2(z)$ is also a constant. Since $P_1(z+c)+P_2(z)$, $P_2(z+c)+P_1(z)$ and $P_1(z)+P_2(z)$ are constants it follows that $P_1(z+c)-P_1(z)$ and $P_2(z+c)-P_2(z)$ are both constants.
Assume that
\begin{align}\label{T1.16} 
	e^{-\iota (P_1(z)+P_2(z+c))}=A\;\;\text{and}\;\;e^{-\iota (P_2(z)+P_1(z+c))}=B.
\end{align}

\begin{figure}[H]
	\centering
	
	\begin{tikzpicture}[
		x=1cm,
		y=1cm,
		every node/.style={font=\small},
		box/.style={
			draw,
			rounded corners=4pt,
			align=center,
			minimum width=4.8cm,
			minimum height=1cm,
			inner sep=4pt
		},
		arr/.style={
			-{Latex[length=2.3mm]},
			thick
		}
		]
		
		
		\node[box] (A1) at (-4,0)
		{$P_1(z+c)+P_2(z)=C_1$};
		
		\node[box] (A2) at (4,0)
		{$P_2(z+c)+P_1(z)=C_2$};
		
		
		\node[box] (B1) at (0,-2)
		{$P_1(z)+P_2(z)=C_3$};
		
		
		\node[box] (C1) at (-4,-4)
		{$P_1(z+c)-P_1(z)=C_4$};
		
		\node[box] (C2) at (4,-4)
		{$P_2(z+c)-P_2(z)=C_5$};
		
		
		\node[box,
		minimum width=7cm,
		minimum height=1.2cm] (D1) at (0,-6.3)
		{Translation structure of $P_1$ and $P_2$};
		
		
		\draw[arr] (A1.south) -- (B1.north);
		
		\draw[arr] (A2.south) -- (B1.north);
		
		\draw[arr] (B1.south west) -- (C1.north);
		
		\draw[arr] (B1.south east) -- (C2.north);
		
		\draw[arr] (C1.south east) -- (D1.north west);
		
		\draw[arr] (C2.south west) -- (D1.north east);
		
	\end{tikzpicture}

	\caption{
		Logical dependence of the shift relations obtained
		from (\ref{T1.15})--(\ref{T1.16}).
	}

	\label{fig:rigidityflow}
	
\end{figure}
\FloatBarrier
Then from (\ref{T1.15}) and (\ref{T1.16}), we have
\begin{align*}
\frac{\partial P_1(z)}{\partial z_1}+1=A.
\end{align*}

Obviously $P_1(z)=(A-1)z_1+\tilde Q_1(z_2,\ldots,z_n)$, where $\tilde Q_1(z_2,\ldots,z_n)$ is a polynomial in $\mathbb{C}^{n-1}$. Since $P_1(z+c)-P_1(z)$ is a constant, it follows that $\tilde Q_1(z_2+c_2,\ldots,z_n+c_n)-\tilde Q_1 (z_2,\ldots,z_n)$ is also a constant. Therefore we may assume that
\begin{align}\label{T1.17}
P_1(z)=(A-1)z_1+A_{12}z_2+\ldots+A_{1n}z_n+\hat Q_1(z_2,\ldots,z_n),
\end{align}
where $A_{1i}\in\mathbb{C}$ for $i=2,\ldots,n$ and $\hat Q_1(z_2,\ldots,z_n)$ is a polynomial in $\mathbb{C}^{n-1}$ such that
$\hat Q_1(z_2+c_2,\ldots,z_n+c_n)=\hat Q_1(z_2,\ldots,z_n)$. Since $P_1(z)+P_2(z)$ is a constant, from (\ref{T1.17}), we have 
\begin{align*}
P_2(z)=-(A-1)z_1-A_{12}z_2-\ldots-A_{1n}z_n-\hat Q_1(z_2,\ldots,z_n)+C_1,
\end{align*}
where $C_1\in\mathbb{C}$. We consider the following two sub-cases.

\medskip
{\bf Sub-case 1.1.} Let $A=1$.  
Now from (\ref{T1.17}), we have $\frac{\partial P_1(z)}{\partial z_1}\equiv 0$ and so $\frac{\partial P_2(z)}{\partial z_1}\equiv 0$. Consequently from (\ref{T1.7}) and (\ref{T1.10}),  we have respectively
\begin{align}\label{T1.13}
	\cos P_2(z+c)\equiv \cos P_1(z)\quad \text{and}\quad \cos P_1(z+c)\equiv \cos P_2(z).
\end{align}

On integration, we get respectively from (\ref{T1.3}) and (\ref{T1.5}) that
\begin{align}\label{T1.14}
	f_1(z)=z_1\cos P_1(z)+g_1(z_2,\ldots,z_n)\; \text{and}\; f_2(z)=z_1\cos P_2(z)+g_2(z_2,\ldots,z_n), 
\end{align}
where $g_1(z_2,\ldots,z_n)$ and $g_2(z_2,\ldots,z_n)$ are finite order entire functions in $\mathbb{C}^{n-1}$.

Now using (\ref{T1.13}) and (\ref{T1.14}) in (\ref{T1.4}), we get 
\begin{align*}
	(z_1+c_1)\cos P_1(z)-z_1\cos P_1(z)+g_2(z_2+c_2,\ldots,z_n+c_n)-g_1(z_2,\ldots,z_n)=\sin P_1(z),
\end{align*}
i.e., 
\begin{align*}
	g_2(z_2+c_2,\ldots,z_n+c_n)-g_1(z_2,\ldots,z_n)=\sin P_1(z)-c_1\cos P_1(z).
\end{align*}
Similarly  using (\ref{T1.13}) and (\ref{T1.14}) in (\ref{T1.6}), we get 
\begin{align*}
	g_1(z_2+c_2,\ldots,z_n+c_n)-g_2(z_2,\ldots,z_n)=\sin P_2(z)-c_1\cos P_2(z).
\end{align*}

Therefore, we have
\begin{align*}
	f_1(z)=z_1\cos P_1(z)+g_1(z_2,\ldots,z_n)\; \text{and}\; f_2(z)=z_1\cos P_2(z)+g_2(z_2,\ldots,z_n), 
\end{align*}
where $g_1(z_2,\ldots,z_n)$ and $g_2(z_2,\ldots,z_n)$ are finite order entire functions in $\mathbb{C}^{n-1}$ such that 
\begin{align*}
	g_i(z_2+c_2,\ldots,z_n+c_n)-g_j(z_2,\ldots,z_n)=\sin P_j(z)-c_1\cos P_j(z),
\end{align*}
where $i,j\in\{1,2\}$ such that $i\neq j$.

\medskip
{\bf Sub-case 1.2.} Let $A\neq 1$. Note that $P_1(z)+P_2(z)=C_1\in\mathbb{C}$. Then from (\ref{T1.3}) and (\ref{T1.5}), we have respectively
\begin{align*}
\frac{\partial f_1(z)}{\partial z_1}=\frac{1}{A-1}\frac{\partial \sin P_1(z)}{\partial z_1}\quad \text{and}\quad \frac{\partial f_2(z)}{\partial z_1}=-\frac{1}{A-1}\frac{\partial \sin (-P_1(z)+C_1)}{\partial z_1}. 
\end{align*}
On integration, we have
\begin{equation}\label{T1.20}
\begin{aligned}
f_1(z)&=\frac{1}{A-1}\sin P_1(z)+\tilde g_1(z_2,\ldots,z_n),\\
f_2(z)&=-\frac{1}{A-1}\sin (-P_1(z)+C_1)+\tilde g_2(z_2,\ldots,z_n).
\end{aligned}
\end{equation}
where $\tilde g_1(z_2,\ldots,z_n)$ and $\tilde g_2(z_2,\ldots,z_n)$ are finite order entire functions in $\mathbb{C}^{n-1}$.
Since $P_1(z+c)-P_1(z)=C_2\in\mathbb{C}$, from (\ref{T1.20}), we have
\begin{align}\label{T1.21}
f_2(z+c)=-\frac{1}{A-1}\sin (-P_1(z)+C_0)+\tilde g_2(z_2+c_2,\ldots,z_n+c_n),
\end{align}
where $C_0=C_1-C_2$. 
Now using (\ref{T1.20}) and (\ref{T1.21}) to (\ref{T1.4}), we obtain
\begin{align*}
-\frac{1}{A-1}\sin (-P_1(z)+C_0)-\frac{A}{A-1}\sin P_1(z)+\tilde g_2(z_2+c_2,\ldots,z_n+c_n)-\tilde g_1(z_2,\ldots,z_n)=0,
\end{align*}
i.e., 
\begin{align}\label{T1.22}
&\left(-\frac{d}{A-1}+\frac{A}{A-1}\right)e^{-iP_1(z)}+\left(\frac{1}{d(A-1)}-\frac{A}{A-1}\right)e^{iP_1(z)}\\&+2\iota\tilde g_2(z_2+c_2,\ldots,z_n+c_n)-2\iota\tilde g_1(z_2,\ldots,z_n)=0,\nonumber
\end{align}
where $d=e^{iC_0}$. 
\begin{figure}[htbp]
	\centering
	
	\begin{tikzpicture}[
		x=0.95cm,
		y=1.15cm,
		>=Latex,
		font=\small
		]
		
		
		\colorlet{waveA}{blue!70!black}
		\colorlet{waveB}{red!70!black}
		\colorlet{shiftcol}{violet!80!black}
		\colorlet{boxcol}{green!50!black}
		
		
		\draw[->,thick]
		(-0.3,0)--(10.8,0)
		node[right] {$z_1$};
		
		\draw[->,thick]
		(0,-1.8)--(0,1.8);
		
		
		\def\Czero{1.3}
		
		
		\draw[
		thick,
		waveA,
		domain=0:10,
		samples=300,
		smooth
		]
		plot (\x,{sin((180/pi)*\x)});
		
		
		\draw[
		thick,
		dashed,
		waveB,
		domain=0:10,
		samples=300,
		smooth
		]
		plot (\x,{sin((180/pi)*(-\x+\Czero))});
		
		
		\node[waveA]
		at (2.1,1.2)
		{$\sin P_1(z)$};
		
		\node[waveB]
		at (6.8,1.2)
		{$\sin(-P_1(z)+C_0)$};
		
		
		\draw[densely dotted,shiftcol]
		(\Czero,0)--(\Czero,-1.15);
		
		\draw[<->,thick,shiftcol]
		(0,-1.0)--(\Czero,-1.0);
		
		\node[shiftcol]
		at ({0.5*\Czero},-1.32)
		{$C_0$};
		
		
		\fill[black] (0,0) circle (1.6pt);
		
		\fill[black] (3.93,0.71) circle (1.6pt);
		
		\fill[black] (7.07,-0.71) circle (1.6pt);
		
		
		\node[
		draw=boxcol,
		fill=green!6,
		rounded corners=5pt,
		align=center,
		text=boxcol,
		minimum width=5cm,
		minimum height=1cm
		]
		at (8.2,-1.35)
		{
			Matching oscillatory phases\\
			force coefficient cancellation
		};
		
		
		\draw[->,thick,boxcol]
		(5.8,0.15)--(6.7,-0.7);
		
	\end{tikzpicture}
	
	\caption{
		Phase interaction in Sub-case~1.2.
		The identity
		\(
		P_1(z+c)-P_1(z)=C_2
		\)
		creates two sinusoidal terms with the same frequency but different phase.
		Comparing these terms gives (\ref{T1.23}) and hence \(A=-1\).
	}
	
	\label{fig:phaseinteraction}
	
\end{figure}
\FloatBarrier
Now from (\ref{T1.22}), we deduce that 
\begin{align}\label{T1.23}
-\frac{d}{A-1}+\frac{A}{A-1}=0\quad \text{and} \quad \frac{1}{d(A-1)}-\frac{A}{A-1}=0
\end{align}
and 
\begin{align}\label{T1.24}
\tilde g_2(z_2+c_2,\ldots,z_n+c_n)\equiv \tilde g_1(z_2,\ldots,z_n).
\end{align}
Clearly from (\ref{T1.23}), we conclude that $d=-1$ and so $A=-1$.  Now, \bea\label{e3.23}f_1(z +c)=\frac{1}{A-1}\sin (P_1(z)+C_2)+\tilde g_1(z_2+c_2,\ldots,z_n+c_n).\eea Using (\ref{T1.20}), (\ref{e3.23}) in (\ref{T1.6}), and comparing the two exponential coefficients, we get $e^{\iota \tilde{C_0}}=-1$ and \bea\label{e3.24} \tilde g_1(z_2+c_2,\ldots,z_n+c_n)-\tilde g_2(z_2,\ldots,z_n)=0,\eea where $\tilde{C_0}=C_1+C_2$.
Finally, from (\ref{T1.17}), (\ref{T1.20}) and (\ref{T1.24}), (\ref{e3.24}) we have
\begin{align*}
f_1(z)=-\frac{1}{2}\sin P(z)+g_1(z_2,\ldots,z_n)\; \text{and}\; f_2(z)=\frac{1}{2}\sin (-P(z)+C)+ g_2(z_2,\ldots,z_n), 
\end{align*}
where $P(z)=-2z_1+A_{12}z_2+\ldots+A_{1n}z_n+ Q(z_2,\ldots,z_n)$, $A_{1i},C\in\mathbb{C}$ for $i=2,\ldots,n$ and $Q(z_2,\ldots,z_n)$ is a polynomial in $\mathbb{C}^{n-1}$ such that
$Q(z_2+c_2,\ldots,z_n+c_n)=Q(z_2,\ldots,z_n)$ and $e^{\iota(P(z)+C-P(z+c))}=-1, \; e^{\iota (-P(z)+C+P(z+c))}=-1$ and $e^{2\iota C }=1$,
$g_1(z_2,\ldots,z_n)$ and $g_2(z_2,\ldots,z_n)$ are finite-order entire functions in $\mathbb{C}^{n-1}$ such that $g_i(z_2+c_2,\ldots,z_n+c_n)\equiv  g_j(z_2,\ldots,z_n)$, $i,j=1,2$ and $i\ne j$.

\medskip
{\bf Case 2.} Let
\begin{align}\label{T1.25} 
\begin{cases} \left( \frac{\partial P_1(z)}{\partial z_1}+1\right)e^{\iota (P_1(z)+P_2(z+c))}=1\\ \left(\frac{\partial P_2(z)}{\partial z_1}+1\right)e^{-\iota (P_2(z)-P_1(z+c))}=1.\end{cases}
\end{align}
	 
Clearly from (\ref{T1.25}), it follows that $P_1(z)+P_2(z+c)$ and $P_2(z)-P_1(z+c)$ are both constants. This means $P_2(z+2c)+P_2(z)$ is also a constant, which contradicts the fact that $P_2(z)$ is a non-constant polynomial.\par

\medskip
{\bf Case 3.} Let
\begin{align}\label{T1.26} 
\begin{cases} \left(\frac{\partial P_1(z)}{\partial z_1}+1\right)e^{-\iota (P_1(z)-P_2(z+c))}=1\\ \left(\frac{\partial P_2(z)}{\partial z_1}+1\right)e^{\iota (P_2(z)+P_1(z+c))}=1.\end{cases}
\end{align}
	 
We deduce from (\ref{T1.26}) that $P_1(z)-P_2(z+c)$ and $P_2(z)+P_1(z+c)$ are both constants. This means $P_1(z+2c)+P_1(z)$ is also a constant, which is a contradiction.\par

\medskip
{\bf Case 4.} Let
\begin{align}\label{T1.27}
\begin{cases} \left( \frac{\partial P_1(z)}{\partial z_1}+1\right)e^{-\iota (P_1(z)-P_2(z+c))}=1\\ \left(\frac{\partial P_2(z)}{\partial z_1}+1\right)e^{-\iota (P_2(z)-P_1(z+c))}=1.\end{cases}
\end{align}
	 
Clearly from (\ref{T1.27}), it follows that $P_1(z)-P_2(z+c)$ and $P_2(z)-P_1(z+c)$ are both constants. Consequently $P_1(z+2c)-P_1(z)$ and $P_2(z+2c) -P_2(z)$ are also constants. Assume that
\begin{align}\label{T1.28} 
e^{\iota (P_1(z)-P_2(z+c))}=\tilde A\;\;\text{and}\;\;e^{\iota (P_2(z)-P_1(z+c))}=\tilde B.
\end{align}
	 
Then from (\ref{T1.27}) and (\ref{T1.28}), we have
$\frac{\partial P_1(z)}{\partial z_1}+1=\tilde A$ and $\frac{\partial P_2(z)}{\partial z_1}+1=\tilde B$.
On integration, we have
$P_1(z)=(\tilde A-1)z_1+\tilde Q_1(z_2,\ldots,z_n)$ and 
$P_2(z)=(\tilde B-1)z_1+\tilde Q_2(z_2,\ldots,z_n)$,
where $\tilde Q_i(z_2,\ldots,z_n)$ is a polynomial in $\mathbb{C}^{n-1}$ for $i=1,2$. Since $P_i(z+2c)-P_i(z)$ is a constant, it follows that $\tilde Q_i(z_2+2c_2,\ldots,z_n+2c_n)-\tilde Q_i(z_2,\ldots,z_n)$ is also a constant for $i=1,2$. Therefore we may assume that
\begin{align}\label{T1.29}\begin{cases}
P_1(z)=(\tilde A-1)z_1+A_{12}z_2+\ldots+A_{1n}z_n+\hat Q_1(z_2,\ldots,z_n),\\
P_2(z)=(\tilde B-1)z_1+B_{12}z_2+\ldots+B_{1n}z_n+\hat Q_2(z_2,\ldots,z_n).
\end{cases}
\end{align}
where $A_{1i},B_{1i}\in\mathbb{C}$ for $i=2,\ldots,n$ and $\hat Q_i(z_2,\ldots,z_n)$ is a polynomial in $\mathbb{C}^{n-1}$ such that
$\hat Q_i(z_2+2c_2,\ldots,z_n+2c_n)=\hat Q_i(z_2,\ldots,z_n)$ for $i=1,2$. Since $P_1(z)-P_2(z+c)$ is a constant, it follows from (\ref{T1.29}) that $\tilde A=\tilde B$. We consider the following two sub-cases.
	 
\medskip
{\bf Sub-case 4.1.} Let $\tilde A=1$. Clearly  $\frac{\partial P_1(z)}{\partial z_1}\equiv 0$ and so $\frac{\partial P_2(z)}{\partial z_1}\equiv 0$. Consequently from (\ref{T1.7}) and (\ref{T1.10}), we arrive at (\ref{T1.13}). Proceeding in the same way as done in the proof of Sub-case 1.1, we obtain
$f_1(z)=z_1\cos P_1(z)+g_1(z_2,\ldots,z_n)$ and $f_2(z)=z_1\cos P_2(z)+g_2(z_2,\ldots,z_n)$, 
where $P_1(z)=A_{12}z_2+\ldots+A_{1n}z_n+Q_1(z_2,\ldots,z_n)$ and $P_2(z)=B_{12}z_2+\ldots+B_{1n}z_n+Q_2(z_2,\ldots,z_n)$,
$A_{1i},B_{1i}\in\mathbb{C}$ for $i=2,\ldots,n$ and $Q_i(z_2,\ldots,z_n)$ is a polynomial in $\mathbb{C}^{n-1}$ such that $Q_i(z_2+2c_2,\ldots,z_n+2c_n)=Q_i(z_2,\ldots,z_n)$ for $i=1,2$, $P_1(z)-P_2(z+c)$ is a constant, $g_1(z_2,\ldots,z_n)$ and $g_2(z_2,\ldots,z_n)$ are finite-order entire functions in $\mathbb{C}^{n-1}$ such that 
\begin{align*}
g_i(z_2+c_2,\ldots,z_n+c_n)-g_j(z_2,\ldots,z_n)=\sin P_j(z)-c_1\cos P_j(z),
\end{align*}
where $i,j\in\{1,2\}$ such that $i\neq j$.
	 
\medskip
{\bf Sub-case 4.2.} Let $\tilde A\neq 1$.
Then from (\ref{T1.3}) and (\ref{T1.5}), we have respectively
\begin{align*}
\frac{\partial f_1(z)}{\partial z_1}=\frac{1}{\tilde A-1}\frac{\partial \sin P_1(z)}{\partial z_1}\quad \text{and}\quad \frac{\partial f_2(z)}{\partial z_1}=\frac{1}{\tilde A-1}\frac{\partial \sin P_2(z)}{\partial z_1}. 
\end{align*}
On integration, we have
\begin{align}\label{T1.30}
f_1(z)=\frac{1}{\tilde A-1}\sin P_1(z)+\tilde g_1(z_2,\ldots,z_n)\; \text{and}\; f_2(z)=\frac{1}{\tilde A-1}\sin P_2(z)+\tilde g_2(z_2,\ldots,z_n), 
\end{align}
where $\tilde g_1(z_2,\ldots,z_n)$ and $\tilde g_2(z_2,\ldots,z_n)$ are finite-order entire functions in $\mathbb{C}^{n-1}$. Since $P_1(z)-P_2(z+c)$ is a constant, from (\ref{T1.30}), we have
\begin{align}\label{T1.31}
f_2(z+c)=\frac{1}{\tilde A-1}\sin (P_1(z)+C_0)+\tilde g_2(z_2+c_2,\ldots,z_n+c_n),
\end{align}
where $C_0\in\mathbb{C}$. Now proceeding in the same way as done in the proof of Sub-case 1.2, we obtain
\begin{align*}
f_1(z)=-\frac{1}{2}\sin P_1(z)+g_1(z_2,\ldots,z_n)\; \text{and}\; f_2(z)=-\frac{1}{2}\sin P_2(z)+ g_2(z_2,\ldots,z_n), 
\end{align*}
where $P_1(z)=-2z_1+A_{12}z_2+\ldots+A_{1n}z_n+ Q_1(z_2,\ldots,z_n)$ and $P_2(z)=-2z_1+B_{12}z_2+\ldots+B_{1n}z_n+ Q_2(z_2,\ldots,z_n)$, $A_{1i},B_{1i}\in\mathbb{C}$ for $i=2,\ldots,n$ and $Q_i(z_2,\ldots,z_n)$ is a polynomial in $\mathbb{C}^{n-1}$ such that
$Q_i(z_2+2c_2,\ldots,z_n+2c_n)=Q_i(z_2,\ldots,z_n)$, $i=1,2$ and $P_1(z)-P_2(z+c)$ is a constant,
$e^{\iota(P_i(z)-P_j(z+c))}=-1$ for $i,j\in\{1,2\}$ with $i\ne j$, and
$g_1(z_2,\ldots,z_n)$ and $g_2(z_2,\ldots,z_n)$ are finite-order entire functions in $\mathbb{C}^{n-1}$ such that
$g_i(z_2+c_2,\ldots,z_n+c_n)\equiv g_j(z_2,\ldots,z_n)$ for $i,j\in\{1,2\}$ with $i\ne j$.

Conversely, substitution of each of the four families, together with its stated phase and shift relations, verifies (\ref{pds1}).
\end{proof}

\section {{\bf Proof of Theorem \ref{t2.2}}}
\begin{proof} Let $(f_1(z),f_2(z))$ be a pair of finite-order entire solutions to the system of equations (\ref{pds1}) in $\mathbb{C}^n$ such that $\frac{\partial f_1(z)}{\partial z_1}$ is non-constant. As in the proof of Lemma~\ref{L.9}, $\frac{\partial f_2(z)}{\partial z_1}$ is also non-constant.
In this case, we have $n_i+m_i>2$ for $i=1,2$, and $n_i\neq m_i$ for at least one $i\in\{1,2\}$. 

We consider the following three cases.\par

\begin{figure}[ht] \centering \resizebox{\textwidth}{!}{%
\begin{tikzpicture}[ node distance=1.2cm and 0.8cm, every node/.style={font=\small}, case/.style={rectangle, draw, rounded corners, minimum width=2.2cm, minimum height=0.7cm, align=center}, subcase/.style={rectangle, draw, minimum width=2cm, minimum height=0.6cm, align=center, font=\footnotesize}, result/.style={rectangle, draw, dashed, minimum width=1.8cm, minimum height=0.5cm, align=center, font=\scriptsize}, arrow/.style={->, >=stealth} ] 
	  \node[case] (start) {$(n_1, m_1, n_2, m_2)$\\with $n_i + m_i > 2$}; \node[case, below left=1.5cm and 2cm of start] (case1) {Case 1\\$n_1 = m_1 = 2$}; \node[case, below=.9cm of start] (case2) {Case 2\\$m_1 > n_1$}; \node[case, below right=1.5cm and 2cm of start] (case3) {Case 3\\$m_1 < n_1$};
	  \node[subcase, below left=2cm and 0.2cm of case1] (sc11) {Sub-case 1.1\\$m_2 < n_2$}; \node[subcase, below right=2cm and 0.2cm of case1] (sc12) {Sub-case 1.2\\$m_2 > n_2$}; 
	  \node[subcase, below left=.65cm and 0.01cm of case2] (sc21) {Sub-case 2.1\\$m_2 \neq n_2$}; \node[subcase, below right=.65cm and 0.01cm of case2] (sc22) {Sub-case 2.2\\$n_2 = m_2 = 2$}; 
		  \node[subcase, below left=2cm and 0.2cm of case3] (sc31) {Sub-case 3.1\\$m_2 \neq n_2$}; \node[subcase, below right=2cm and 0.2cm of case3] (sc32) {Sub-case 3.2\\$n_2 = m_2 = 2$}; 
			  \draw[arrow] (start) -- (case1); \draw[arrow] (start) -- (case2); \draw[arrow] (start) -- (case3); \draw[arrow] (case1) -- (sc11); \draw[arrow] (case1) -- (sc12); \draw[arrow] (case2) -- (sc21); \draw[arrow] (case2) -- (sc22); \draw[arrow] (case3) -- (sc31); \draw[arrow] (case3) -- (sc32); \end{tikzpicture}} \caption{Case analysis in the proof of Theorem \ref{t2.2}; the branches record the exponent relations.} \label{fig:proof-structure} \end{figure}

\medskip
{\bf Case 1.} Let $n_1= m_1$. Clearly $n_2\neq m_2$. Note that $n_1+m_1>2$ and so (\ref{lm.5}) gives $n_1=m_1=2$. Since $m_1=n_1=2$, from (\ref{pds1}),  we get
\begin{align}\label{lm.7} 
\left(\frac{\partial f_1(z)}{\partial z_1}\right)^2+\left(f_2(z+c)-f_1(z)\right)^2=1.
\end{align}

Now we consider the following two sub-cases.\par

\medskip
{\bf Sub-case 1.1.} Let $m_2<n_2$. In this case, from (\ref{lm.5}), we obtain $m_2=1$ and $n_2>1$. Then from (\ref{pds1}), we get
\begin{align}\label{lm.8} 
\left(\frac{\partial f_2(z)}{\partial z_1}\right)^{n_2}+f_1(z+c)-f_2(z)=1.
\end{align}

Clearly (\ref{lm.8}) gives
\begin{align*}
 n_2\left(\frac{\partial f_2(z)}{\partial z_1}\right)^{n_2-1}\frac{\partial^2 f_2(z)}{\partial z_1^2}+\frac{\partial f_1(z+c)}{\partial z_1}-\frac{\partial f_2(z)}{\partial z_1}=0
 \end{align*}
and so from (\ref{lm.7}), we get
\begin{align}\label{lm.8a}
\left(\frac{\partial f_2(z)}{\partial z_1}-n_2\left(\frac{\partial f_2(z)}{\partial z_1}\right)^{n_2-1}\frac{\partial^2 f_2(z)}{\partial z_1^2}\right)^2+\left(f_2(z+2c)-f_1(z+c)\right)^2=1.
\end{align}

Using (\ref{lm.8}) to (\ref{lm.8a}), we obtain
\begin{align}\label{lm.9}
\left(\frac{\partial f_2(z)}{\partial z_1}-n_2\left(\frac{\partial f_2(z)}{\partial z_1}\right)^{n_2-1}\frac{\partial^2 f_2(z)}{\partial z_1^2}\right)^2+\left(f_2(z+2c)+\left(\frac{\partial f_2(z)}{\partial z_1}\right)^{n_2}-f_2(z)-1\right)^2=1.
\end{align}

Applying Theorem D to (\ref{lm.9}), we obtain
\begin{align}\label{lm.10} 
\frac{\partial f_2(z)}{\partial z_1}-n_2\left(\frac{\partial f_2(z)}{\partial z_1}\right)^{n_2-1}\frac{\partial^2 f_2(z)}{\partial z_1^2}=\cos h(z)
\end{align}
and
\begin{align}\label{lm.11} 
f_2(z+2c)+\left(\frac{\partial f_2(z)}{\partial z_1}\right)^{n_2}-f_2(z)-1=\sin h(z),
\end{align}
where $h:\mathbb{C}^n\to \mathbb{C}$ is an entire function. Note that $\rho\left(\frac{\partial f_2(z)}{\partial z_1}\right)<+\infty$ and so $\rho\left(\frac{\partial^2 f_2(z)}{\partial z_1^2}\right)<+\infty$. Now using Lemmas \ref{L.3} and \ref{L.5} to (\ref{lm.10}), we conclude that $h(z)$ is a polynomial in $\mathbb{C}^n$. Clearly from (\ref{lm.11}), we obtain
\begin{align*}
\frac{\partial f_2(z+2c)}{\partial z_1}+n_2\left(\frac{\partial f_2(z)}{\partial z_1}\right)^{n_2-1}\frac{\partial^2 f_2(z)}{\partial z_1^2}-\frac{\partial f_2(z)}{\partial z_1}=\frac{\partial h(z)}{\partial z_1}\cos h(z)
\end{align*}
and so from (\ref{lm.10}), we get
\begin{align}\label{lm.12}
\frac{\partial f_2(z+2c)}{\partial z_1}=\frac{\partial h(z)}{\partial z_1}\cos h(z)+\cos h(z).
\end{align}

Therefore from (\ref{lm.10}) and (\ref{lm.12}), we have
\begin{align} &\label{lm.13}n_2\left(1+\frac{\partial h(z)}{\partial z_1}\right)^{n_2-1}\left(\frac{\partial^2 h(z)}{\partial z_1^2}\cos h(z)-\left(1+\frac{\partial h(z)}{\partial z_1}\right)\frac{\partial h(z)}{\partial z_1}\sin h(z)\right)\cos^{n_2-1} h(z)\nonumber\\&=\left(1+\frac{\partial h(z)}{\partial z_1}\right) \cos h(z)-\cos h(z+2c).
\end{align}

Clearly $\frac{\partial h(z)}{\partial z_1}\not\equiv 0,-1$. Note that the sets of multiple zeros of both $\cos h(z)$ and $\cos h(z+2c)$ are algebraic. Let
\begin{align}\label{lm.13a}
H(z)=\frac{\left(1+\frac{\partial h(z)}{\partial z_1}\right) \cos h(z)-\cos h(z+2c)}{\cos h(z)}.
\end{align}

Note that
\begin{align*}
\cos h(z)=\frac{e^{2\iota h(z)}+1}{2e^{\iota h(z)}}
\end{align*}
and so by Lemma \ref{L.3}, we get $T(r,\cos h(z))=2T\left(r, e^{\iota h(z)}\right)+O(1)$. Thereby $o(T(r,\cos h(z)))$ can be replaced by $o\left(T\left(r, e^{\iota h(z)}\right)\right)$.

\medskip
First we suppose that $H\equiv 0$. Then from (\ref{lm.13}), we have 
\begin{align*}
\frac{\partial^2 h(z)}{\partial z_1^2}\cos h(z)-\left(1+\frac{\partial h(z)}{\partial z_1}\right)\frac{\partial h(z)}{\partial z_1}\sin h(z)\equiv 0,
\end{align*}
which shows that $\frac{\partial^2 h(z)}{\partial z_1^2}\not\equiv 0$ and so
\begin{align*}
\left(\frac{\partial^2 h(z)}{\partial z_1^2}+i\left(1+\frac{\partial h(z)}{\partial z_1}\right)\frac{\partial h(z)}{\partial z_1}\right)e^{\iota h(z)}+\left(\frac{\partial^2 h(z)}{\partial z_1^2}-i\left(1+\frac{\partial h(z)}{\partial z_1}\right)\frac{\partial h(z)}{\partial z_1}\right)\equiv 0,
\end{align*}
which implies that
\begin{align*}
\frac{\partial^2 h(z)}{\partial z_1^2}+i\left(1+\frac{\partial h(z)}{\partial z_1}\right)\frac{\partial h(z)}{\partial z_1}\equiv 0\;\;\text{and}\;\;\frac{\partial^2 h(z)}{\partial z_1^2}-i\left(1+\frac{\partial h(z)}{\partial z_1}\right)\frac{\partial h(z)}{\partial z_1}\equiv 0.
\end{align*}

Consequently, we have $\left(1+\frac{\partial h(z)}{\partial z_1}\right)\frac{\partial h(z)}{\partial z_1}\equiv 0$, which is absurd.

\medskip
Next we suppose $H\not\equiv 0$.
Now from (\ref{lm.13}), we see that if $z_0\in\mathbb{C}^n$ is a zero of $\cos h(z)$ of multiplicity $p_0$, then $z_0$ is a zero of $\left(1+\frac{\partial h(z)}{\partial z_1}\right) \cos h(z)-\cos h(z+2c)$ of multiplicity at least $(n_2-1)p_0\geq 1$. Since the sets of multiple zeros of both $\cos h(z)$ and $\cos h(z+2c)$ are algebraic, from (\ref{lm.13a}), we deduce that $N(r,H)=O(\log r)$. Using Lemma \ref{L.8} to (\ref{lm.13a}), we obtain
\begin{align*}
T(r,H)=N(r,H)+m(r,H)=O(\log r)+o(T(r,\cos h(z)))=o(T(r,\cos h(z))),
\end{align*}
which shows that $H$ is a small function of $\cos h(z)$. Now from (\ref{lm.13}) and (\ref{lm.13a}), we get
\begin{align}\label{lm.13b} n_2\left(1+\frac{\partial h(z)}{\partial z_1}\right)^{n_2-1}\left( \frac{\partial^2 h(z)}{\partial z_1^2}\cos h(z)-\left(1+\frac{\partial h(z)}{\partial z_1}\right)\frac{\partial h(z)}{\partial z_1}\sin h(z)\right)\cos^{n_2-2} h(z)=H(z).
\end{align}

We consider the following two sub-cases.

\medskip
{\bf Sub-case 1.1.1.} Let $n_2\geq 3$. Clearly the set of multiple zeros of $\cos^{n_2-1} h(z)$ is not algebraic. Also from (\ref{lm.13b}), we deduce that $N(r,0;\cos h(z))=o(T(r,\cos h(z)))$.  Since $N(r,0;\cos h(z))=o(T(r,\cos h(z)))$, it follows that $N\left(r,1;e^{2\iota h(z)}\right)=o\left(T\left(r, e^{\iota h(z)}\right)\right)$. Now using Lemma \ref{L.2} for the function $e^{2\iota h(z)}$,  we get a contradiction. 

\medskip
{\bf Sub-case 1.1.2.} Let $n_2=2$. Then from (\ref{lm.13b}), we have
\begin{align} \label{lm.14} 
2\left(1+\frac{\partial h(z)}{\partial z_1}\right) \frac{\partial^2 h(z)}{\partial z_1^2} \cos h(z)-2\left(1+\frac{\partial h(z)}{\partial z_1}\right)^2\frac{\partial h(z)}{\partial z_1} \sin h(z)=H(z),
\end{align}
where $H(z)$ is a small function of $\cos h(z)$. Since $T(r,\cos h(z))=2T\left(r, e^{\iota h(z)}\right)+O(1)$, it follows that $H(z)$ is a small function of $e^{\iota h(z)}$.

\medskip
First we suppose that $\frac{\partial^2 h(z)}{\partial z_1^2}\equiv 0$. Clearly from (\ref{lm.14}), we get
\begin{align} \label{lm.14a} 
-2\left(1+\frac{\partial h(z)}{\partial z_1}\right)^2\frac{\partial h(z)}{\partial z_1}\left(e^{2ih(z)}-1\right)=2iH(z)e^{\iota h(z)}.
\end{align}

Since $H(z)$ is a small function of $e^{\iota h(z)}$, it follows from (\ref{lm.14a}) that $N\left(r,1;e^{2\iota h(z)}\right)=o\left(T\left(r, e^{\iota h(z)}\right)\right)$. Now using Lemma \ref{L.2} for the function $e^{2\iota h(z)}$,  we get a contradiction. 

\medskip
Next we suppose that $\frac{\partial^2 h(z)}{\partial z_1^2}\not\equiv 0$. By simple calculations on (\ref{lm.14}), we get
\begin{align}\label{lm.15} 
\tilde G_1(z)e^{2ih(z)}-H(z)e^{\iota h(z)}+\tilde G_2(z)\equiv 0,
\end{align}
where
\begin{align*}
\tilde G_1(z)=\left(1+\frac{\partial h(z)}{\partial z_1}\right) \frac{\partial^2 h(z)}{\partial z_1^2}+i\left(1+\frac{\partial h(z)}{\partial z_1}\right)^2\frac{\partial h(z)}{\partial z_1}
\end{align*}
and
\begin{align*}
\tilde G_2(z)=-\left(\left(1+\frac{\partial h(z)}{\partial z_1}\right) \frac{\partial^2 h(z)}{\partial z_1^2}-i\left(1+\frac{\partial h(z)}{\partial z_1}\right)^2\frac{\partial h(z)}{\partial z_1}\right).
\end{align*}

Since $\tilde G_1(z)$, $\tilde G_2(z)$ and $H(z)\not\equiv 0$ are small functions of $e^{\iota h(z)}$, using Lemma \ref{L.3} to (\ref{lm.15}), we get a contradiction.

\medskip
{\bf Sub-case 1.2.} Let $m_2>n_2$. In this case, from (\ref{lm.5}), we obtain $m_2>1$ and $n_2=1$. Then from (\ref{pds1}), we get
\begin{align}\label{lmm.1} 
\frac{\partial f_2(z)}{\partial z_1}+\left(f_1(z+c)-f_2(z)\right)^{m_2}=1.
\end{align}

Now using Theorem D to (\ref{lm.7}), we get
\begin{align}\label{lmm.2} 
\frac{\partial f_1(z)}{\partial z_1}=\cos h(z)
\end{align}
and
\begin{align}\label{lmm.3} 
f_2(z+c)-f_1(z)=\sin h(z),
\end{align}
where $h:\mathbb{C}^n\to \mathbb{C}$ is a polynomial. Now using (\ref{lmm.2}) to (\ref{lmm.3}), we get
\begin{align}\label{lmm.4}
\frac{\partial f_2(z+c)}{\partial z_1}=\left(1+\frac{\partial h(z)}{\partial z_1}\right) \cos h(z).
\end{align}

Again using (\ref{lmm.4}) to (\ref{lmm.1}), we obtain
\begin{align}\label{lmm.5}
\left(f_1(z+2c)-f_2(z+c)\right)^{m_2}=1-\left(1+\frac{\partial h(z)}{\partial z_1}\right) \cos h(z).
\end{align}

First we suppose
$\frac{\partial h(z)}{\partial z_1}\equiv -1$. Clearly $h(z)$ depends upon the variable $z_1$ and so from (\ref{lmm.2}), we conclude that $f_1(z)$ and $\frac{\partial f_1(z)}{\partial z_1}$ must depend upon the variable $z_1$. On the other hand, from (\ref{lmm.4}), we deduce that $f_2(z+c)$ must be independent of the variable $z_1$. Now differentiating (\ref{lmm.5}) partially with respect to $z_1$ and then using (\ref{lmm.2}) and (\ref{lmm.4}), we get
\begin{align*}
m_2\left(f_1(z+2c)-f_2(z+c)\right)^{m_2-1}\cos h(z+2c)\equiv 0,
\end{align*}
which shows that $f_1(z+2c)\equiv f_2(z+c)$. Since $f_1(z)$ depends on the variable $z_1$, whereas $f_2(z+c)$ is independent of the variable $z_1$, we get a contradiction.

\medskip
Next we suppose $\frac{\partial h(z)}{\partial z_1}\not\equiv -1$. Clearly from (\ref{lmm.5}), we have
\begin{align}\label{lmm.6}
\left(f_1(z+2c)-f_2(z+c)\right)^{m_2}=-\frac{g(z)}{2e^{\iota h(z)}},
\end{align}
where 
\begin{align}\label{lmm.7}
g(z)=\left(1+\frac{\partial h(z)}{\partial z_1}\right)e^{2\iota h(z)}-2e^{i h(z)}+\left(1+\frac{\partial h(z)}{\partial z_1}\right).
\end{align}

Obviously
\begin{align}\label{lmm.8}
\frac{\partial g(z)}{\partial z_1}=\left(\frac{\partial^2 h(z)}{\partial z_1^2}+2i \frac{\partial h(z)}{\partial z_1}\left(1+\frac{\partial h(z)}{\partial z_1}\right)\right)e^{2ih(z)}-2i\frac{\partial h(z)}{\partial z_1} e^{\iota h(z)}+\frac{\partial^2 h(z)}{\partial z_1^2}.
\end{align}

\medskip
Suppose $\frac{\partial h(z)}{\partial z_1}\equiv 0$. Clearly $h(z)$ is independent upon the variable $z_1$. On integration, we get from (\ref{lmm.2}) that $f_1(z)=z_1\cos h(z)+\phi(z_2,z_3,\ldots,z_n)$, where $\phi(z_2,z_3,\ldots,z_n)$ is a finite-order entire function in $\mathbb{C}^{n-1}$. Clearly from (\ref{lmm.3}), we have $f_2(z+c)=z_1\cos h(z)+\sin h(z)+\phi(z_2,z_3,\ldots,z_n)$. Now using (\ref{lmm.7}) to (\ref{lmm.6}), we obtain
\begin{align*}
\left(f_1(z+2c)-f_2(z+c)\right)^{m_2}=2\sin^2 \frac{1}{2}h(z),
\end{align*}	 
which shows that $m_2=2$ and so
$f_1(z+2c)-f_2(z+c)=\pm \sqrt{2}\sin \frac{1}{2}h(z)$. Since 
$f_1(z)=z_1\cos h(z)+\phi(z_2,z_3,\ldots,z_n)$ and $f_2(z+c)=z_1\cos h(z)+\sin h(z)+\phi(z_2,z_3,\ldots,z_n)$, we have 
\begin{align*}
&(z_1+2c_1)\cos h(z+2c)+\phi(z_2+2c_2,\ldots,z_n+2c_n)-
z_1\cos h(z)-\sin h(z)-\phi(z_2,\ldots,z_n)\\&=\pm \sqrt{2}\sin \frac{1}{2}h(z).
\end{align*}

\medskip
Suppose $\frac{\partial h(z)}{\partial z_1}\not\equiv 0$. Let $z_0\in\mathbb{C}^n$ be a zero of $g(z)$ such that $z_0$ is not a zero of the coefficients of the right hand sides of $g(z)$ and $\frac{\partial g(z)}{\partial z_1}$.
Since $m_2\geq 2$, from (\ref{lmm.6}), we deduce that $z_0$ is multiple zero of $g(z)$ and so $z_0$ is also a zero of $\frac{\partial g(z)}{\partial z_1}$. Clearly 
\begin{align}\label{lmm.9}
\left(1+\frac{\partial h(z_0)}{\partial z_1}\right)e^{2i h(z_0)}-2e^{i h(z_0)}+\left(1+\frac{\partial h(z_0)}{\partial z_1}\right)=0
\end{align}
and
\begin{align}\label{lmm.10}
\left(\frac{\partial^2 h(z_0)}{\partial z_1^2}+2i \frac{\partial h(z_0)}{\partial z_1}\left(1+\frac{\partial h(z_0)}{\partial z_1}\right)\right)e^{2ih(z_0)}-2i\frac{\partial h(z_0)}{\partial z_1} e^{ih(z_0)}+\frac{\partial^2 h(z_0)}{\partial z_1^2}=0.
\end{align}

Eliminating $e^{2ih(z_0)}$ from (\ref{lmm.9}) and (\ref{lmm.10}), we obtain
\begin{align}\label{lmm.11}
\left(\frac{\partial^2 h(z_0)}{\partial z_1^2}+i\frac{\partial h(z_0)}{\partial z_1}\left(1+\frac{\partial h(z_0)}{\partial z_1}\right)\right)e^{ih(z_0)}=i\frac{\partial h(z_0)}{\partial z_1}\left(1+\frac{\partial h(z_0)}{\partial z_1}\right)^2.
\end{align}
	
Again eliminating $e^{ih(z_0)}$ from (\ref{lmm.9}) and (\ref{lmm.10}), we obtain
\begin{align}\label{lmm.12}
\left(\frac{\partial^2 h(z_0)}{\partial z_1^2}+i\frac{\partial h(z_0)}{\partial z_1}\left(1+\frac{\partial h(z_0)}{\partial z_1}\right)\right)e^{2ih(z_0)}=-\left(\frac{\partial^2 h(z_0)}{\partial z_1^2}-i\frac{\partial h(z_0)}{\partial z_1}\left(1+\frac{\partial h(z_0)}{\partial z_1}\right)\right).
\end{align}

Finally from (\ref{lmm.11}) and (\ref{lmm.12}), we get
\begin{align}\label{lmm.13}
\left(\frac{\partial^2 h(z_0)}{\partial z_1^2}\right)^2=\left(\frac{\partial h(z_0)}{\partial z_1}\right)^3 \left(1+\frac{\partial h(z_0)}{\partial z_1}\right)^2\left(\frac{\partial h(z_0)}{\partial z_1}+2\right).
\end{align}
	
\medskip
First suppose $\left(\frac{\partial^2 h(z)}{\partial z_1^2}\right)^2\not\equiv \left(\frac{\partial h(z)}{\partial z_1}\right)^3 \left(1+\frac{\partial h(z)}{\partial z_1}\right)^2\left(\frac{\partial h(z)}{\partial z_1}+2\right)$. Then from (\ref{lmm.13}), we can deduce that the set of zeros of $g(z)$ is algebraic. If we use Lemmas \ref{L.2} and \ref{L.3} to (\ref{lmm.7}), then we immediately get a contradiction.
	
\medskip
Next suppose $\left(\frac{\partial^2 h(z)}{\partial z_1^2}\right)^2\equiv \left(\frac{\partial h(z)}{\partial z_1}\right)^3 \left(1+\frac{\partial h(z)}{\partial z_1}\right)^2\left(\frac{\partial h(z)}{\partial z_1}+2\right)$. Since $\frac{\partial h(z)}{\partial z_1}\not\equiv 0,-1$, one can easily conclude that $\frac{\partial h(z)}{\partial z_1}\equiv -2$. Now from (\ref{lmm.2}), we have 
\begin{align*}
\frac{\partial f_1(z)}{\partial z_1}=\frac{\partial (-\frac{1}{2} \sin h(z))}{\partial z_1}
\end{align*}
and so on integration, we have $f_1(z)=-\frac{1}{2}\sin h(z)+\psi(z_2,z_3,\ldots,z_n)$, where $\psi(z_2,z_3,\ldots,z_n)$ is a finite-order entire function in $\mathbb{C}^{n-1}$. Therefore from (\ref{lmm.3}), we have $f_2(z+c)=\frac{1}{2}\sin h(z)+\psi(z_2,z_3,\ldots,z_n)$. Now using (\ref{lmm.7}) to (\ref{lmm.6}), we obtain
\begin{align*}
\left(f_1(z+2c)-f_2(z+c)\right)^{m_2}=2\cos^2 \frac{1}{2}h(z),
\end{align*}	 
which shows that $m_2=2$ and so
$f_1(z+2c)-f_2(z+c)=\pm \sqrt{2}\cos \frac{1}{2}h(z)$. Since 
$f_1(z)=-\frac{1}{2}\sin h(z)+\psi(z_2,z_3,\ldots,z_n)$ and $f_2(z+c)=\frac{1}{2}\sin h(z)+\psi(z_2,z_3,\ldots,z_n)$, we have 
\begin{align*}
-\frac{1}{2}\sin h(z+2c)+\psi(z_2+2c_2,\ldots,z_n+2c_n)
-\frac{1}{2}\sin h(z)-\psi(z_2,\ldots,z_n)=\pm \sqrt{2}\cos \frac{1}{2}h(z).
\end{align*}

\medskip
{\bf Case 2.} Let $m_1>n_1$. In this case, from (\ref{lm.5}), we obtain $m_1>1$ and $n_1=1$ and so from (\ref{pds1}), we get
\begin{align}\label{lnm.1} 
F_1(z)+G_1^{m_1}(z)=1,
\end{align}
where $F_1$ and $G_1$ are defined in (\ref{lm}).
We now consider the following sub-cases.

\medskip
{\bf Sub-case 2.1.} Let $m_2\neq n_2$. In this case, $G_i(z)$ and $f_j(z)$ share $0$ CM and $\rho(G_i/f_j)<1$, where $i,j\in\{1,2\}$ such that $i\neq j$. Clearly 
\begin{align*}
G_i(z)=e^{P_j(z)}f_j(z),
\end{align*}
where $P_j(z)$ is a polynomial in $\mathbb{C}^n$ and $i,j\in\{1,2\}$ such that $i\neq j$. Now using Lemma \ref{L.5}, we deduce that $P_j(z)$ is a constant. Therefore, we may assume that 
\begin{align}\label{lnm.2}
G_i(z)=d_jf_j(z),
\end{align}
where $i,j\in\{1,2\}$ such that $i\neq j$. We now consider the following sub-cases.

\medskip
{\bf Sub-case 2.1.1.} Let $n_2>m_2$ such that $n_2>2$ and $m_1>\frac{n_2}{n_2-2}$. Clearly (\ref{lm.5}) implies $n_2>1$ and $m_2=1$ and so from (\ref{pds1}), we get
\begin{align}\label{lnm.3} 
F_2^{n_2}(z)+G_2(z)=1.
\end{align}

In view of (\ref{lnm.2}) and by a simple calculation, (\ref{lnm.1}) and (\ref{lnm.3}) give
\begin{align}\label{lnm.4}
-n_2F_2^{n_2-1}(z)\frac{\partial F_2(z)}{\partial z_1}+d_1d_2^{m_1}f_2^{m_1}(z)=1.
\end{align}

Let $\phi(z)=-n_2F_2^{n_2-1}(z)\frac{\partial F_2(z)}{\partial z_1}$. Now using Lemmas \ref{L.1}, \ref{L.3} and \ref{L.8a}, we have
\begin{align*}
 T(r,F_2^{n_2}(z))=m(r,F_2^{n_2}(z))&=m\left(r,\phi(z)\frac{F_2(z)}{\frac{\partial F_2(z)}{\partial z_1}}\right)\\&\leq
m(r,\phi(z))+m\left(r,\frac{F_2(z)}{\frac{\partial F_2(z)}{\partial z_1}}\right)\\&\leq
T(r,\phi(z))+T\left(r,\frac{F_2(z)}{\frac{\partial F_2(z)}{\partial z_1}}\right)-N\left(r,\frac{F_2(z)}{\frac{\partial F_2(z)}{\partial z_1}}\right)+O(1)\\&\leq
T(r,\phi(z))+T\left(r,\frac{\frac{\partial F_2(z)}{\partial z_1}}{F_2(z)}\right)-N\left(r,\frac{F_2(z)}{\frac{\partial F_2(z)}{\partial z_1}}\right)+O(1)\\&\leq
T(r,\phi)+N\left(r,\frac{\frac{\partial F_2(z)}{\partial z_1}}{F_2(z)}\right)-N\left(r,\frac{F_2(z)}{\frac{\partial F_2(z)}{\partial z_1}}\right)+o(T(r,F_2))\\&\leq
T(r,\phi)+N(r,0;F_2)-N\left(r,0;\frac{\partial F_2(z)}{\partial z_1}\right)+o(T(r,F_2)),\nonumber
\end{align*}
i.e.,
\begin{align}\label{lnm.5}
 (n_2-1)T(r,F_2)\leq T(r,\phi)-N\left(r,0;\frac{\partial F_2(z)}{\partial z_1}\right)+o(T(r,F_2)).
 \end{align}

Therefore using Lemma \ref{L.2}, (\ref{lnm.4}) to (\ref{lnm.5}), we obtain
\begin{align}\label{lnm.6} 
(n_2-1)T(r,F_2)&\leq \ol N(r,0;\phi)+\ol N(r,1;\phi)-N\left(r,0;\frac{\partial F_2(z)}{\partial z_1}\right)+o(T(r,F_2))\\&\leq\ol N(r,0;F_2)+\ol N(r,0;f_2(z+2c))+o(T(r,F_2))\nonumber\\&\leq
T(r,F_2)+T(r,f_2(z+2c))+o(T(r,F_2)).\nonumber
\end{align}

On the other hand, applying Lemmas \ref{L.1} and \ref{L.3} to (\ref{lnm.4}), we get
\begin{align}\label{lnm.7} 
m_1 T(r,f_2)=T(r,\phi)+O(1)=m(r,\phi)+O(1)&\leq n_2 m(r,F_2)+o(T(r,F_2))\\&=n_2 T(r,F_2)+o(T(r,F_2)).\nonumber
\end{align}

Consequently (\ref{lnm.6}) and (\ref{lnm.7}) yield
\begin{align*}
\left(n_2-2-\frac{n_2}{m_1}\right)T(r,F_2)\leq o(T(r,F_2)),
\end{align*}
which contradicts the fact that $m_1>\frac{n_2}{n_2-2}$.\par

\medskip
{\bf Sub-case 2.1.2.} Let $n_2<m_2$. Clearly (\ref{lm.5}) implies $n_2=1$ and $m_2>1$ and so from (\ref{pds1}), we get
\begin{align}\label{lnm.8} 
F_2(z)+G_2^{m_2}(z)=1.
\end{align}
Now using (\ref{lnm.2}) to (\ref{lnm.1}) and (\ref{lnm.8}), we get respectively
\begin{align}\label{lnm.9} 
F_1(z)+d_2f_2^{m_1}(z)=1\quad \text{and}\quad F_2(z)+d_1f_1^{m_2}(z)=1.
\end{align}
Using Lemmas \ref{L.1} and \ref{L.3} to (\ref{lnm.9}), we obtain
\begin{align}\label{lnm.10}
\parallel\;m_i T(r, f_i)=T\left(r, f_i^{m_i}\right)+O(1) =T\left(r,1-F_j\right)+O(1)
& =T(r,F_j)+O(1)\\
& =m(r,F_j)+O(1) \nonumber\\
& \leq \left(m\left(r, \frac{F_j}{f_j}\right)+m(r, f_j)\right)+O(1)\nonumber \\
& =T(r, f_j)+O(T(r, f_j)),\nonumber
\end{align}
where $i,j\in\{1,2\}$ such that $i\neq j$. From  (\ref{lnm.10}), one can easily deduce that
\begin{align*}
\parallel\;(m_1m_2-1)T(r,f_i)\leq o(T(r, f_i)),
\end{align*}
which is impossible, where $i\in\{1,2\}$. Therefore the system of Fermat-type partial differential-difference equations (\ref{pds1}) does not have any finite-order transcendental entire solutions.

\medskip
{\bf Sub-case 2.2.} Let $n_2=m_2$. Then (\ref{lm.5}) gives $n_2=m_2=2$ and so from (\ref{pds1}), we get
\begin{align}\label{lnm.11} 
F_2^2(z)+G_2^2(z)=1.
\end{align}

Now using Theorem D to (\ref{lnm.11}), we obtain
\begin{align}\label{lnm.12} 
F_2(z)=\cos h(z)
\end{align}
and
\begin{align}\label{lnm.13} 
G_2(z)=f_1(z+c)-f_2(z)=\sin h(z)
\end{align}
where $h:\mathbb{C}^n\to \mathbb{C}$ is a polynomial. Now proceeding in the same way as done in the proof of Sub-case 1.2, we conclude that one of the following cases holds.
\begin{enumerate}
\item[(i)]
\begin{align*}
f_2(z)&=z_1\cos h(z)+\phi(z_2,z_3,\ldots,z_n),\\
f_1(z+c)&=z_1\cos h(z)+\sin h(z)+\phi(z_2,z_3,\ldots,z_n).
\end{align*}
Here $h(z)$ is a polynomial in $\mathbb{C}^n$ such that $\frac{\partial h(z)}{\partial z_1}\equiv 0$ and
\begin{align*}
&-(z_1+2c_1)\cos h(z+2c)-\phi(z_2+2c_2,\ldots,z_n+2c_n)+
z_1\cos h(z)+\sin h(z)+\phi(z_2,\ldots,z_n)\\&=\pm \sqrt{2}\sin \frac{1}{2}h(z);
\end{align*}
\item[(ii)]
\begin{align*}
f_2(z)&=-\frac{1}{2}\sin h(z)+\psi(z_2,z_3,\ldots,z_n),\\
f_1(z+c)&=\frac{1}{2}\sin h(z)+\psi(z_2,z_3,\ldots,z_n).
\end{align*}
Here $h(z)$ is a polynomial in $\mathbb{C}^n$ such that $\frac{\partial h(z)}{\partial z_1}\equiv -2$ and
\begin{align*}
\frac{1}{2}\sin h(z+2c)-\psi(z_2+2c_2,\ldots,z_n+2c_n)+
\frac{1}{2}\sin h(z)+\psi(z_2,\ldots,z_n)=\pm \sqrt{2}\cos \frac{1}{2}h(z).
\end{align*}
\end{enumerate}

\medskip
{\bf Case 3.} Let $m_1<n_1$. In this case from (\ref{lm.5}), we obtain $m_1=1$ and $n_1>1$ and so from (\ref{pds1}), we get
\begin{align}\label{lnn.1} 
F_1^{n_1}(z)+G_1(z)=1.
\end{align}

We now consider following sub-cases.

\medskip
{\bf Sub-case 3.1.} Let $m_2\neq n_2$. In this case, $G_i(z)$ and $f_j(z)$ share $0$ CM and $\rho(G_i/f_j)<1$, where $i,j\in\{1,2\}$ such that $i\neq j$. Clearly (\ref{lnm.2}) holds. We consider the following sub-cases.

\medskip
{\bf Sub-case 3.1.1.} Let $n_2<m_2$ such that $n_1>2$ and $m_2>\frac{n_1}{n_1-2}$. Clearly (\ref{lm.5}) implies $n_2=1$ and $m_2>1$ and so from (\ref{pds1}), we get
\begin{align}\label{lnn.2}
F_2(z)+G_2^{m_2}(z)=1.
\end{align}

In view of (\ref{lnm.2}) and by a simple calculation, (\ref{lnn.1}) and (\ref{lnn.2}) give
\beas-n_1F_1^{n_1-1}(z)\frac{\partial F_1(z)}{\partial z_1}+d_1^{m_2}d_2f_1^{m_2}(z)=1.\eeas

Since $m_2>\frac{n_1}{n_1-2}$, proceeding in the same way as done in the proof of Sub-case 2.1.1, we get a contradiction.\par

\medskip
{\bf Sub-case 3.1.2.} Let $n_2>m_2$. Clearly (\ref{lm.5}) implies $n_2>1$ and $m_2=1$ and so from (\ref{pds1}), we get
\begin{align}\label{lnn.3} 
F_2^{n_2}(z)+G_2(z)=1.
\end{align}

Using (\ref{lnm.2}) to (\ref{lnn.1}) and (\ref{lnn.3}), we get
\begin{align}\label{lnn.4}
F_i^{n_i}(z)+d_jf_j(z)=1,
\end{align}
where $i,j\in\{1,2\}$ such that $i\neq j$.
Differentiating (\ref{lnn.4}) partially with respect to $z_1$, we get
\begin{align}\label{lnn.5} 
n_i F_i^{n_i-1}(z)\frac{\partial F_i(z)}{\partial z_1}+d_jF_j(z)=0,
\end{align}
where $i,j\in\{1,2\}$ such that $i\neq j$.
Clearly (\ref{lnn.4}) and (\ref{lnn.5}), we obtain
\begin{align}\label{lnn.6}
\left(-n_iF_i^{n_i-1}(z)\frac{\partial F_i(z)}{\partial z_1}\right)^{n_j}+d_id_j^{n_j}f_i(z)=1,
\end{align}
where $i,j\in\{1,2\}$ such that $i\neq j$.
Differentiating (\ref{lnn.6}) partially with respect to $z_1$, we get
\begin{align}\label{lnn.7} 
F_i^{n_1n_2-n_j-1}(z)\alpha_i(z)=-d_id_j^{n_j}F_i(z),
\end{align}
where
\begin{align}\label{lnn.8}
\alpha_i(z)= (-n_i)^{n_j}n_j\left(\frac{\partial F_i(z)}{\partial z_1}\right)^{n_j-1}\left[(n_i-1)\left(\frac{\partial F_i(z)}{\partial z_1}\right)^2+F_i(z)\frac{\partial^2 F_i(z)}{\partial z_1^2}\right]
\end{align}
for $i,j\in\{1,2\}$ such that $i\neq j$.
Clearly $\alpha_i\not\equiv 0$, otherwise from (\ref{lnn.7}), we get $F_i(z)\equiv 0$, which is impossible for $i=1,2$. Since $n_i>1$, we have $n_1n_2-n_i\geq 2$ for $i=1,2$. Therefore applying Lemma \ref{L.4} to (\ref{lnn.7}), we get $\parallel m(r,\alpha_i)=o(T(r,F_i))$ and so $\parallel T(r,\alpha_i)=o(T(r,F_i))$ for $i=1,2$.
Now consider the following sub-cases.\par

\medskip
{\bf Sub-case 3.1.2.1.} Let $n_1n_2-n_i>2$ for $i=1,2$. Now using Lemma \ref{L.4} to (\ref{lnn.7}), we get $\parallel m(r,F_i\alpha)=o(T(r,F_i))$ and so $\parallel T(r,F_i\alpha)=o(T(r,F_i))$ for $i=1,2$. Since $\parallel T(r,\alpha_i)=o(T(r,F_i))$, it follows that
\begin{align*}
\parallel\;T(r,F_i)\leq T(r,F_i\alpha_i)+T\left(r,\frac{1}{\alpha_i}\right)=T(r,F_i\alpha_i)+T(r,\alpha_i)=o(T(r,F_i)),
\end{align*}
which is impossible for $i=1,2$.\par

\medskip
{\bf Sub-case 3.1.2.2.} Let $n_1n_2-n_i=2$ for $i=1,2$. In this case, we must have $n_1=n_2=2$. Then from (\ref{lnn.7}) and (\ref{lnn.8}), we have respectively
\begin{align}\label{lnn.9} 
\alpha_i(z)=-d_id_j^2
\end{align}
for $i,j\in\{1,2\}$ such that $i\neq j$ and
\begin{align}\label{lnn.10}\alpha_i(z)=8\frac{\partial F_i(z)}{\partial z_1}\left[\left(\frac{\partial F_i(z)}{\partial z_1}\right)^2+F_i(z)\frac{\partial^2 F_i(z)}{\partial z_1^2}\right]=8\frac{\partial F_i(z)}{\partial z_1}\frac{\partial \left(F_i(z)\frac{\partial F_i(z)}{\partial z_1}\right)}{\partial z_1}
\end{align}
for $i=1,2$. Clearly (\ref{lnn.9}) and (\ref{lnn.10}) give
\begin{align}\label{lnn.11}
\alpha_i(z)=8\frac{\partial F_i(z)}{\partial z_1}\left[\left(\frac{\partial F_i(z)}{\partial z_1}\right)^2+F_i(z)\frac{\partial^2 F_i(z)}{\partial z_1^2}\right]=8\frac{\partial F_i(z)}{\partial z_1}\frac{\partial \left(F_i(z)\frac{\partial F_i(z)}{\partial z_1}\right)}{\partial z_1}=-d_id_j^2
\end{align}
for $i,j\in\{1,2\}$ such that $i\neq j$. Now differentiating (\ref{lnn.11}) partially with respect to $z_1$, we get
\begin{align}\label{lnn.12} 
4\left(\frac{\partial F_i(z)}{\partial z_1}\right)^2\frac{\partial^2 F_i(z)}{\partial z_1^2}+F_i(z)\left( \left(\frac{\partial^2 F_i(z)}{\partial z_1^2}\right)^2+\frac{\partial F_i(z)}{\partial z_1}\frac{\partial^3 F_i(z)}{\partial z_1^3}\right)=0
\end{align}
for $i=1,2$. Note that $\parallel T(r,\alpha_i)=o(T(r,F_i))$ for $i=1,2$. Therefore from (\ref{lnn.11}), it is easy to verify that
\begin{align}\label{lnn.13} 
\parallel\;N\left(r,0;\frac{\partial F_i(z)}{\partial z_1}\right)\leq N(r,0;\alpha_i)\leq T(r,\alpha_i)=o(T(r,F_i))
\end{align}
for $i=1,2$.
We consider the following two sub-cases.\par

\medskip
{\bf Sub-case 3.1.2.2.1.} Let $\frac{\partial^2 F_i(z)}{\partial z_1^2}\not\equiv 0$ for $i=1,2$. Now in view of the first main theorem and using Lemma \ref{L.1}, we get
\begin{align}\label{lnn.14} 
\parallel\;N\left(r,0;\frac{\partial^2 F_i(z)}{\partial z_1^2}\right)&\leq N\left(r,0;\frac{\frac{\partial^2 F_i(z)}{\partial z_1^2}}{\frac{\partial F_i(z)}{\partial z_1}}\right)+N\left(r,0;\frac{\partial F_i(z)}{\partial z_1}\right)\nonumber\\&\leq
T\left(r,\frac{\frac{\partial^2 F_i(z)}{\partial z_1^2}}{\frac{\partial F_i(z)}{\partial z_1}}\right)+N\left(r,0;\frac{\partial F_i(z)}{\partial z_1}\right)+O(1)\nonumber\\&\leq
N\left(r,\frac{\partial F_i(z)}{\partial z_1}\right)+2N\left(r,0;\frac{\partial F_i(z)}{\partial z_1}\right)+o(T(r,F_i))
\end{align}
for $i=1,2$. Since $F_i(z)$ is an entire function, from (\ref{lnn.13}) and (\ref{lnn.14}), we obtain
\begin{align}\label{lnn.15}
\parallel\;N\left(r,0;\frac{\partial^2 F_i(z)}{\partial z_1^2}\right)=o(T(r,F_i))
\end{align}
for $i=1,2$. Therefore using (\ref{lnn.13}) and (\ref{lnn.15}) to (\ref{lnn.12}), we deduce that
\begin{align}\label{lnn.16}
\parallel N(r,0;F_i)=o(T(r,F_i))
\end{align}
for $i=1,2$.
Now (\ref{lnn.11}) gives
\begin{align*}
\frac{\alpha_i(z)}{F_i^3(z)}=8\left(\frac{\frac{\partial F_i(z)}{\partial z_1}}{F_i(z)}\right)^3+8\frac{\frac{\partial F_i(z)}{\partial z_1}}{F_i(z)}\times\frac{\frac{\partial^2 F_i(z)}{\partial z_1^2}}{F_i(z)}
\end{align*}
and so by Lemma \ref{L.1}, we get
\begin{align}\label{lnn.17}
\parallel\;m\left(r,0;F_i(z)\right)=o(T(r,F_i))
\end{align}
for $i=1,2$. Now in view of the first main theorem and using (\ref{lnn.16}) and (\ref{lnn.17}), we obtain
$\parallel T(r,F_i)=o(T(r,F_i))$ for $i=1,2$, which is impossible.\par

\medskip
{\bf Sub-case 3.1.2.2.2.} Let $\frac{\partial^2 F_i(z)}{\partial z_1^2}\equiv 0$ for $i=1,2$. Now from (\ref{lnn.10}), we have
\begin{align}\label{lnn.18}\alpha_i(z)F_i(z)=4\frac{\partial \left(F_i(z)\frac{\partial F_i(z)}{\partial z_1}\right)^2}{\partial z_1}
\end{align}
for $i=1,2$. On the other hand, from (\ref{lnn.5}) and (\ref{lnn.18}), we obtain
\begin{align*}
F_i(z)=\frac{2d_j}{\alpha_i(z)}F_j(z)\frac{\partial F_j(z)}{\partial z_1}
\end{align*}
and so (\ref{lnn.5}) yields
\begin{align}\label{lnn.19} 
F_i(z)\frac{\partial F_i(z)}{\partial z_1}+\frac{d_id_j}{\alpha_j(z)}F_i(z)\frac{\partial F_i(z)}{\partial z_1}=0
\end{align}
for $i,j\in\{1,2\}$ such that $i\neq j$. Clearly from (\ref{lnn.19}) gives $\alpha_j(z)=-d_id_j$ for $i,j\in\{1,2\}$ such that $i\neq j$.
Therefore $\alpha_i(z)=\alpha_j(z)=-d_id_j$ and so from (\ref{lnn.9}), we get $d_1=d_2=1$. Consequently $\alpha_i(z)=-1$ for $i=1,2$.
Since $\frac{\partial^2 F_i(z)}{\partial z_1^2}\equiv 0$, equations (\ref{lnn.10})--(\ref{lnn.11}) reduce to $8(\partial F_i/\partial z_1)^3=-1$. Hence
\begin{align}\label{lnn.20} 
\frac{\partial F_i(z)}{\partial z_1}=\frac{K_i}{2},
\end{align}
where $K_i^3=-1$ for $i=1,2$. Clearly (\ref{lnn.20}) implies
\begin{align}\label{lnn.21} 
\frac{\partial f_i(z)}{\partial z_1}=F_i(z)=\frac{K_i}{2}z_1+g_i(z_2,\ldots,z_n),
\end{align}
where $g_i(z_2,z_3,\ldots, z_n)$ is a finite-order transcendental entire function for $i=1,2$.  Again from (\ref{lnn.21}), we have
\begin{align}\label{lnn.22} 
f_i(z)=\frac{K_i}{4}z_1^2 + z_1g_i(z_2,z_3,\ldots,z_n)+G_i(z_2,z_3,\ldots,z_n),
\end{align}
where $G_i(z_2,z_3,\ldots, z_n)$ is a finite-order transcendental entire function for $i=1,2$.
Using (\ref{lnn.20}), (\ref{lnn.21}) and (\ref{lnn.22}) to (\ref{lnn.6}), we get
\begin{align*} &K_i^2\left(\frac{K_i^2}{4}z_1^2+K_iz_1g_i(z_2,\ldots,z_n)+g_i^2(z_2,\ldots,z_n)\right)\\&+
\frac{K_i}{4}z_1^2 + z_1g_i(z_2,z_3,\ldots,z_n)+G_i(z_2,z_3,\ldots,z_n)=1,
\end{align*}
i.e., $G_i(z_2,z_3,\ldots,z_n)=1-K_i^2(g_i(z_2,z_3,\ldots,z_n))^2$ for $i=1,2$
and so from (\ref{lnn.22}), we have
\begin{align*}
 f_i(z)=1+\frac{K_i}{4}z_1^2 + z_1g_i(z_2,z_3,\ldots,z_n)-K_i^2g_i^2(z_2,z_3,\ldots,z_n),
 \end{align*}
where $K_i^3=-1$ and $g_i(z_2,z_3,\ldots, z_n)$ is a finite-order transcendental entire function. However, (\ref{lnm.2}) and $d_1=d_2=1$ give
\[
f_2(z+c)-f_1(z)=f_2(z).
\]
Comparing the coefficients of $z_1^2$ in this identity and using the preceding representation yields
\[
\frac{K_2-K_1}{4}=\frac{K_2}{4},
\]
so $K_1=0$, contradicting $K_1^3=-1$. Therefore this sub-case admits no finite-order transcendental entire solutions.\par

\medskip
{\bf Sub-case 3.2.} Let $m_2=n_2$. Then (\ref{lm.5}) gives $n_2=m_2=2$ and so from (\ref{pds1}), we get
\begin{align}\label{md.1} 
F_2^2(z)+G_2^2(z)=1.
\end{align}

Now using Theorem D to (\ref{md.1}), we arrive at (\ref{lnm.12}) and (\ref{lnm.13}).
Clearly (\ref{lnn.1}) gives
\begin{align*} 
n_1\left(\frac{\partial f_1(z)}{\partial z_1}\right)^{n_1-1}\frac{\partial^2 f_1(z)}{\partial z_1^2}+\frac{\partial f_2(z+c)}{\partial z_1}-\frac{\partial f_1(z)}{\partial z_1}=0
\end{align*}
and so from (\ref{md.1}), we get
\begin{align}\label{md.2}\left(\frac{\partial f_1(z)}{\partial z_1}-n_1\left(\frac{\partial f_1(z)}{\partial z_1}\right)^{n_1-1}\frac{\partial^2 f_1(z)}{\partial z_1^2}\right)^2+\left(f_1(z+2c)-f_2(z+c)\right)^2=1.
\end{align}

Again using (\ref{lnn.1}) to (\ref{md.2}), we obtain
\begin{align*}
\left(\frac{\partial f_1(z)}{\partial z_1}-n_1\left(\frac{\partial f_1(z)}{\partial z_1}\right)^{n_1-1}\frac{\partial^2 f_1(z)}{\partial z_1^2}\right)^2+\left(f_1(z+2c)+\left(\frac{\partial f_1(z)}{\partial z_1}\right)^{n_1}-f_1(z)-1\right)^2=1.
\end{align*}

Now proceeding in the same way as in the proof of Sub-case 1.1 we get a contradiction.

\end{proof}

\vspace{0.1in}
{\bf Compliance of Ethical Standards:}\par

{\bf Conflict of Interest.} The authors declare that there is no conflict of interest regarding the publication of this paper.\par

{\bf Funding.} The authors have not received any funds for this research.\par

{\bf Data availability statement.} Data sharing not applicable to this article as no data sets were generated or analysed during the current study.

\end{document}